\documentclass[reqno,a4paper, 11pt]{amsart}

\usepackage[a4paper=true,pdfpagelabels]{hyperref}
\usepackage{graphicx}

\usepackage[ansinew]{inputenc}
\usepackage{amsfonts,epsfig}
\usepackage{latexsym}
\usepackage{mathabx}
\usepackage{amsmath}
\usepackage{amssymb}

\usepackage{color}

\newtheorem{theorem}{Theorem}
\newtheorem{lemma}[theorem]{Lemma}
\newtheorem{corollary}[theorem]{Corollary}
\newtheorem{proposition}[theorem]{Proposition}

\newtheorem{lettertheorem}{Theorem}
\newtheorem{letterlemma}[lettertheorem]{Lemma}

\theoremstyle{definition}

\theoremstyle{remark}

\numberwithin{equation}{section}

\newcommand{\abs}[1]{\lvert#1\rvert}
\newcommand{\nm}[1]{\lVert#1\rVert}

\newcommand{\D}{\mathbb{D}}
\newcommand{\DD}{\widehat{\mathcal{D}}}
\newcommand{\Dd}{\widecheck{\mathcal{D}}}
\newcommand{\M}{\mathcal{M}}
\newcommand{\SSS}{\mathcal{S}}
\newcommand{\DDD}{\mathcal{D}}

\newcommand{\N}{\mathbb{N}}

\newcommand{\C}{\mathbb{C}}

\renewcommand{\phi}{\varphi}

\def\Area{\mathord{\rm Area}}
\def\HL{\mathord{\rm HL}}

\def\a{\alpha}       \def\b{\beta}        \def\g{\gamma}
           
     \def\om{\omega}      
       \def\t{\theta}       
         \def\r{\rho}         \def\z{\zeta}
                  \def\vp{\varphi}

\def\omg{\widehat{\omega}}

\def\U{{\mathcal U}}

\renewcommand{\H}{\mathcal{H}}

\begin{document}

\title[Norm inequalities for weighted Dirichlet spaces]{Norm inequalities for weighted Dirichlet spaces with applications to conformal maps}

\keywords{Bergman space, conformal map, Dirichlet space, doubling weight, Hardy space, norm inequality, univalent function}

\author{Fernando P\'erez-Gonz\'alez}
\address{Universidad de La Laguna,  Departament of Mathematical Analysis and Instituto de Matematicas
y Aplicaciones (IMAULL), P.O. Box 456, 38200 La Laguna, Tenerife, Spain}
\email{fpergon@ull.edu.es}

\author{Jouni R\"atty\"a}
\address{University of Eastern Finland, P.O. Box 111, 80101 Joensuu, Finland}
\email{jouni.rattya@uef.fi}

\author{Toni Vesikko}
\address{University of Eastern Finland, P.O. Box 111, 80101 Joensuu, Finland}
\email{tonive@uef.fi}

\thanks{This research was supported in part by Academy of Finland project no.~356029}

\begin{abstract}
A variety of norm inequalities related to Bergman and Dirichlet spaces induced by radial weights is established. Some of the results obtained can be considered as generalizations of certain known special cases while most of the estimates discovered are completely new. In particular, a Littlewood-Paley estimate recently proved by Pel\'aez and the second author [\emph{Bergman projection induced by radial weight, Adv. Math. 391 (2021). Paper No. 107950, 70 pp.}] is improved in part. The second objective of the paper is to apply the obtained norm inequalities to relate the growth of the maximum modulus of a conformal map~$f$, measured in terms of a weighted integrability condition, to a geometric quantity involving the area of image under $f$ of a disc centered at the origin. Our findings in this direction yield new geometric characterizations of conformal maps in certain weighted Dirichlet and Besov spaces.
\end{abstract}

\maketitle

\renewcommand{\thefootnote}{}
\footnotetext[1]{\emph{Mathematics Subject Classification 2020:}
Primary 30H10, 30H20, 30C55; Secondary 30C45.}

\section{Introduction and main results}

Let $\H(\D)$ denote the space of analytic functions in the unit disc $\D=\{z\in\C:|z|<1\}$. For $0<p\le\infty$, the Hardy space $H^p$ consists of those $f\in\H(\D)$ such that
	$$
	\|f\|_{H^p}=\sup_{0<r<1}M_p(r,f)<\infty,
	$$
where
	\begin{equation*}
	\begin{split}
	M_p(r,f)=\left(\frac{1}{2\pi}\int_0^{2\pi}|f(re^{i\theta})|^p\,d\theta\right)^\frac1p, \quad 0<r<1,
	\end{split}
	\end{equation*}
is the $L^p$-mean of the restriction of $f$ to the circle of radius $r$, and $M_{\infty}(r,f)=\max_{|z|=r}|f(z)|$ is the maximum modulus function. The monographs \cite{Duren70} and \cite{Garnett} are excellent sources for the theory of the Hardy spaces.

A function $\nu:\D\to[0,\infty)$ such that $\int_0^1\nu(r)\,dr<\infty$ and $\nu(z)=\nu(|z|)$ for all $z\in\D$ is called a radial weight. For $0<p<\infty$ and such a~$\nu$, the weighted Bergman space $A^p_\nu$ consists of $f\in\H(\D)$ such that
	$$
	\|f\|_{A^p_\nu}^p=\int_{\D}|f|^p\nu\,dA<\infty,
	$$
where $dA(z)=d\theta r\,dr$ for $z=re^{i\theta}$ denotes the element of the Lebesgue area measure on $\D$.
The corresponding weighted Dirichlet space $D^p_\nu$ is equipped with the norm	
	$$
	\|f\|_{D^p_\nu}^p=\|f'\|_{A^p_\nu}^p+|f(0)|^p.
	$$
Throughout this paper we assume $\widehat{\nu}(z)=\int_{|z|}^1\nu(t)\,dt>0$ for all $z\in\D$, for otherwise $A_{\nu}^{p}=\H(\D)=D^p_\nu$. For the theory of weighted Bergman spaces we refer to \cite{DurSchus,HKZ,P-R2014}.

A radial weight $\omega$ belongs to the class $\DD$ if there exists a constant $C=C(\om)>1$ such that the tail integral $\widehat{\om}$ satisfies the doubling condition
	$$
	\widehat{\omega}(r)\le C\widehat{\omega}\left(\frac{1+r}{2}\right),\quad 0\le r<1.
	$$
Furthermore, if there exist $K=K(\omega)>1$ and $C=C(\omega)>1$ such that
	\begin{equation}\label{Eq:def-of-D-check}
	\begin{split}
	\omg(r)\ge C\omg\left(1-\frac{1-r}{K}\right),\quad 0\le r<1,
	\end{split}
	\end{equation}
then we write $\omega\in\Dd$. It is easy to see that \eqref{Eq:def-of-D-check} is equivalent to
	$$
	\widehat{\om}(r)\le\left(1+\frac1{C-1}\right)\int_r^{1-\frac{1-r}{K}}\om(t)\,dt,\quad 0\le r<1,
	$$
which perhaps explains the behavior of weights in $\Dd$ in a more transparent way. The definitions of both $\DD$ and $\Dd$ have their obvious geometric interpretations. The intersection $\DD\cap\Dd$ is denoted by $\DDD$. Further, $\om\in\M$ if there exist $K=K(\omega)>1$ and $C=C(\omega)>1$ such that
	$$
	\om_x\ge C\om_{Kx},\quad x\ge1.
	$$
Here and from now on $\om_x=\int_0^1r^x\om(r)\,dr$ for all $1\le x<\infty$. Further, as usual, we write $A(x)\lesssim B(x)$ or $B(x)\gtrsim A(x)$ for all $x$ in some set $I$ if there exists a constant $C>0$ such that $A(x)\le C B(x)$ for all $x\in I$. Further, if $A(x)\lesssim B(x)\lesssim A(x)$ in $I$, we say that $A(x)$ and $B(x)$ are comparable and write $A(x)\asymp B(x)$ for all $x\in I$, or simply $A\asymp B$.

Observe that a radial weight $\om$ belongs to $\DD$ if and only if $\om_x\asymp\widehat{\om}\left(1-\frac1x\right)$ for all $1\le x<\infty$ by \cite[Lemma~2.1]{Pelaez2016}, and further, the containment in $\DD$ can also be characterized by the moment condition $\om_x\lesssim\om_{2x}$, see Lemma~\ref{D-hat-lemma}(iv) in Section~\ref{sec:thm1}. However, despite the fact that $\DDD=\DD\cap\Dd=\DD\cap\M$~\cite[Theorem~3]{PR2021}, the classes $\Dd$ and $\M$ do not obey the same analogue because $\Dd\subsetneq\M$ by \cite[Proposition~14]{PR2021}. Further, it is known that these classes emerge naturally in many instances in operator theory related to Bergman spaces. For example, they are intimately related to Littlewood-Paley estimates, bounded Bergman projections and the Dostani\'c problem~\cite{PR2021}, as well as the Carleson measures~\cite{LR,LRW,PelRatEmb,P-R2014}. For basic properties and illuminating examples of weights in these and other related classes of weights, see \cite{Pelaez2016,PR2021,P-R2014} and the references therein.

The purpose of this paper is two fold. First, we consider a variety of norm inequalities related to the Bergman and Dirichlet spaces induced by radial weights. Some of the results obtained can be considered as generalizations of certain known special cases while others are completely new. The second objective of the paper is to apply the obtained norm inequalities in order to relate the growth of the maximum modulus of a conformal map $f$, measured in terms of a weighted integrability condition, to a geometric quantity involving the area of image under $f$ of a disc centered at the origin. Our findings in this direction yield geometric characterizations of conformal maps in certain weighted Dirichlet and Besov spaces.

Our first set of results concern weighted Bergman spaces. Later on we will apply these results to study conformal maps in weighted Dirichlet spaces. For convenience, for each $x\in\mathbb{R}$ and a weight $\nu$, we write $\nu_{[x]}(z)=\nu(z)(1-|z|)^x$ for all $z\in\D$.

\begin{theorem}\label{theorem:norm}
Let $\om$ be a radial weight and $0<p,q<\infty$. Then there exists a constant $C=C(p,q,\om)>0$ such that
	\begin{equation}\label{Eq:norm-Bergman}
	\|f\|_{A^p_{\om_{[q-1]}}}\le C\|f\|_{A^p_{\widehat{\om}_{[q-2]}}},\quad f\in\H(\D),
	\end{equation}
if and only if either $\widehat{\om}_{[q-2]}\in\DD$ or $\widehat{\om}_{[q-2]}\not\in L^1$.
\end{theorem}

The case $q=1$ is valid by \cite[Theorem~8]{PR2020}, so our contribution consists of treating the other positive values of $q$. Moreover, for each $1<q<\infty$, we have $\widehat{\om}_{[q-2]}\in\DD$ if and only if $\om\in\DD$ by Lemma~\ref{D-hat-lemma}(v). This equivalence is certainly false in general for $0<q\le1$, because, for each $1<\alpha<\infty$, the weight
	\begin{equation}\label{V}
	v_\alpha(z)=\frac{1}{(1-|z|)\left(\log\frac{e}{1-|z|}\right)^\alpha},\quad z\in\D,
	\end{equation}
belongs to $\DD$, but $\widehat{v_\alpha}_{[-1]}$ is not even a weight if $\alpha\le2$. However, if $\om\in\DD$ and $\widehat{\om}_{[q-2]}\in L^1$, then $\widehat{\om}_{[q-2]}\in\DD$ for each $0<q<\infty$ by the proof of Lemma~\ref{D-hat-lemma}(v). The converse implication is also false because there exists a weight $\om\not\in\DD$ such that $\widehat{\om}_{[-1]}\in\DD$ by \cite[Theorem~3]{PR2020}.

The estimate in Theorem~\ref{theorem:norm} would not be much of use unless we were able to say when the two norms or seminorms are actually comparable. The next theorem establishes the norm comparability we are after.

\begin{theorem}\label{thm-iff}
Let $\om$ be a radial weight and $0<p,q<\infty$ such that $\widehat{\om}_{[q-2]}\in L^1$. Then the following statements hold:
\begin{itemize}
\item[\rm(i)] If $0<q<1$, then
	\begin{equation}\label{norm-inequality-asymp}
	\|f\|_{A^p_{\widehat{\om}_{[q-2]}}}\asymp\|f\|_{A^p_{\om_{[q-1]}}},\quad f\in\H(\D),
	\end{equation}
if and only if $\widehat{\om}_{[q-2]}\in\DD$.
\item[\rm(ii)] If $1\le q<\infty$, then \eqref{norm-inequality-asymp} is satisfied if and only if $\om\in\DDD$.
\end{itemize}
\end{theorem}

A natural interpretation of Theorem~\ref{thm-iff} is that one may replace $\om$ by the weight
	$$
	\widetilde{\om}(z)=\frac{\widehat{\om}(z)}{1-|z|},\quad z\in\D,
	$$
in the weighted Bergman space norm
	$$
	\|f\|_{A^p_{\om_{[q-1]}}}^p=\int_\D|f(z)|^p(1-|z|)^{q-1}\om(z)\,dA(z),\quad f\in\H(\D),
	$$
without any essential loss of information, provided $\om$ and $q$ satisfy the corresponding requirement in the two different cases of the theorem. Further, if also $\widetilde{\om}$ satisfies the same requirement, then one may apply the theorem again to deduce that $\omega$ can be replaced by the weight
	$$
	\widetilde{\widetilde{\om}}(z)=\frac{\widehat{\widetilde{\om}}(z)}{1-|z|}=\frac{\int_{|z|}^1\frac{\widehat{\om}(t)}{1-t}\,dt}{1-|z|},\quad z\in\D,
	$$
which is differentiable and strictly positive. Therefore one may assume $\om$ to obey such smoothness properties which are in general not valid in the classes of weights considered, and that is certainly a very useful property in many concrete calculations.

In view of Theorems~\ref{theorem:norm} and~\ref{thm-iff} it is natural to ask when the norm inequality converse to \eqref{Eq:norm-Bergman} is valid. This question is answered in the case $p=2$ by the following result.

\begin{theorem}\label{thm:A2-weight-M}
Let $\om$ be a radial weight and $1\le q<\infty$. Then the following statements are equivalent:
\begin{itemize}
\item[\rm(i)] There exists a constant $C=C(\om,q)>0$ such that
	\begin{equation}
	\|f\|_{A^2_{\widehat{\om}_{[q-2]}}}\le C \|f\|_{A^2_{\om_{[q-1]}}},\quad f\in\H(\D);
	\end{equation}
\item[\rm(ii)] There exists a constant $C=C(\om,q)>0$ such that
	\begin{equation}
	\left(\widehat{\om}_{[q-2]}\right)_x\le C \left(\om_{[q-1]}\right)_x,\quad 1\le x<\infty;
	\end{equation}
\item[\rm(iii)] $\om\in\M$.
\end{itemize}
\end{theorem}

Proposition~\ref{proposition-norm}(ii) shows that, for each radial weight $\om$, $1\le q<\infty$ and $0<p<\infty$, a necessary condition for the estimate
	\begin{equation}\label{norm-inequality-}
	\|f\|_{A^p_{\widehat{\om}_{[q-2]}}}\lesssim\|f\|_{A^p_{\om_{[q-1]}}},\quad f\in\H(\D),
	\end{equation}
to hold is $\om\in\M$. We have been unable to judge if $\om\in\M$ is also a sufficient condition for \eqref{norm-inequality-} to hold unless $p=2$ (or $p=2n$ for some $n\in\N$). Observe that, for each fixed $0<q<1$, the estimate \eqref{norm-inequality-} is valid for all radial weights by Proposition~\ref{proposition-norm}(i) below.

Theorems~\ref{theorem:norm},~\ref{thm-iff} and~\ref{thm:A2-weight-M} are proved in Sections~\ref{sec:thm1},~\ref{sec:thm2} and~\ref{sec:thm3} in the numerical order. All the proofs are strongly based on good understanding of the weight classes involved in the statements. For convenience of the reader and also for further reference, we have built the proofs in such a way that results which can be considered purely weight related or otherwise possibly useful as such are given in separate lemmas which often show more than what is strictly needed for proving the theorems.

We have one more norm inequality left but before presenting it we give a brief motivation for our study on conformal maps. An injective function in $\H(\D)$ is called a conformal map or univalent, and the class of all such functions is denoted by $\U$. Let $\SSS$ denote the set of $f\in\U$ normalized such that $f(0)=0$ and $f'(0)=1$. If $f\in\U$, then $(f-f(0))/f'(0)$ certainly belongs to $\SSS$. We refer to \cite{Duren83}, \cite{Pommerenke1} and \cite{Pommerenke2} for the theory of univalent functions.

Prawitz'~\cite{Prawitz1927} estimate
		\begin{equation*}
		M_p^p(r,f)\le p\int_0^rM_{\infty}^p(t,f)\frac{dt}{t},\quad 0<r<1,\quad 0<p<\infty,\quad f\in\SSS,
		\end{equation*}
combined with the Hardy-Littlewood \cite[p.~411]{H-L1932} inequality
	\begin{equation}\label{eq:hl}
  \int_0^rM_{\infty}^p(t,f)\,dt\le\pi rM_p^p(r,f),\quad 0<r<1,\quad 0<p<\infty,\quad f\in\mathcal{H}(\D),
  \end{equation}
the proof of which can be found in \cite[p.~841]{B-G-P2004}, \cite[Theorem~2]{NT} or \cite[Hilfssatz~1]{Po61/62}, shows that the well-known asymptotic equality
	\begin{equation}\label{H^p-1}
	\|f\|_{H^p}^p\asymp\int_0^1M_{\infty}^p(t,f)\,dt,\quad f\in\U,
	\end{equation}
is valid for each fixed $0<p<\infty$. Therefore the containment of $f\in\U$ in the Hardy space $H^p$ is neatly characterized by this integrability condition on its maximum modulus.

By applying \eqref{H^p-1} to $z\mapsto f_r(z)=f(rz)$, and then integrating over $[0,1)$ with respect to $\nu(r)r\,dr$, we obtain through Fubini's theorem the estimate
	\begin{equation}\label{Eq:thm-Bergman-2}
	\|f\|_{A_{\nu}^p}^{p}
	\asymp\int_0^1 M_\infty^p(r,f)\left(\int_r^1\nu(t)t\,dt\right)\,dr,\quad f\in\U.
	\end{equation}
This is the natural counterpart of \eqref{H^p-1} for weighted Bergman spaces induced by radial weights.

In view of \eqref{H^p-1} and \eqref{Eq:thm-Bergman-2} it is natural to ask if there is a general phenomenon behind these asymptotic equalities. To give an affirmative answer to this question we set
	\begin{equation*}
	\begin{split}
	J_{\om}^{p}(f)=\int_0^1 M_{\infty}^{p}(r,f)\om(r)\,dr,\quad f\in\H(\D).
	\end{split}
	\end{equation*}
The choice $\omega\equiv1$ gives the quantity appearing on the right hand side of \eqref{H^p-1} while \eqref{Eq:thm-Bergman-2} is covered by the choice $\om(r)=\int_r^1\nu(t)t\,dt$. The connection between $H^p$ and $A^p_\nu$ is also clear through the Hardy-Spencer-Stein identity
	\begin{equation}\label{eq:hpnorm}
  \begin{split}
  \|f\|_{H^p}^p
  &=\frac{1}{2\pi}\int_{\D}\Delta|f|^p(z)\log\frac{1}{|z|}dA(z)+|f(0)|^p,\quad f\in\H(\D),
  \end{split}
  \end{equation}
where $\Delta|f|^p=p^2|f|^{p-2}|f'|^2$ stands for the Laplacian of $|f|^p$. Namely, by applying \eqref{eq:hpnorm} to $f_r$ and integrating with respect to $\nu(r)r\,dr$ yields an analogous identity for the weighted Bergman space~\cite[Theorem~4.2]{P-R2014}:
	\begin{equation}\label{eq:Apnorm}
	\|f\|_{A^p_\nu}^p=\int_\D\Delta|f|^p\nu^\star\,dA+\nu(\D)|f(0)|^p,\quad f\in\H(\D),
	\end{equation}
where
	$$
	\nu^\star(z)=\int_{|z|}^1\log\frac{s}{|z|}\nu(s)s\,ds,\quad z\in\D\setminus\{0\},
	$$
and $\nu(E)=\int_E\nu\,dA$ for each measurable set $E\subset\D$. To understand the big picture better, we define
	\begin{equation}\label{Hclassdef}
	\|f\|_{H_{\om}^{p}}^p
	=\int_0^1\left(\int_{D(0,r)}\Delta|f|^p\,dA\right)\om(r)\,dr\\
	=\int_\D\Delta|f|^p\widehat{\om}\,dA,\quad f\in\H(\D),
	\end{equation}
where $D(0,r)=\{z\in\C:|z|<r\}$. It is clear that $H^p_\om$ for $\om\equiv1$ coincides with the Hardy space $H^p$ by \eqref{eq:hpnorm} while for $\om(r)=\int_{r}^1\nu(t)t\,dt$ we have $H^p_\om=A^p_\nu$ by \eqref{eq:Apnorm}. Our next result establishes a norm inequality involving~$H^p_\om$ and certain weighted Dirichlet spaces.

\begin{theorem}\label{thm:Hp-Ap-S}
Let $0<p<\infty$ and let $\om$ be a radial weight. Then the following statements hold:
\begin{itemize}
\item[\rm(i)] If $0<p<2$, then $\|f\|_{H^p_\om}\lesssim\|f'\|_{A^p_{\widehat{\om}_{[p-2]}}}$ for all $f\in\H(\D)$ if and only if $\widehat{\om}_{[p-2]}\in\DD$ or $\widehat{\om}_{[p-2]}\not\in L^1$;
\item[\rm(ii)] If $2<p<\infty$, then $\|f'\|_{A^p_{\widehat{\om}_{[p-2]}}}\lesssim\|f\|_{H^p_\om}$ for all $f\in\H(\D)$ if and only if $\om\in\DD$.
\end{itemize}
Moreover, both norm estimates are valid for all $f\in\SSS$ without any hypotheses on $\omega$.
\end{theorem}

A reader who is familiar with the Hardy spaces immediately realizes that for $\om\equiv1$ the statement of the theorem reduces to the well-known embeddings $D^p_{p-1}\subset H^p$ and $H^q\subset D^q_{q-1}$ due to Littlewood and Paley~\cite[Theorems~5 and 6]{LP1934}, valid for $0<p\le2\le q<\infty$. Here as usual, $f\in D^p_{p-1}$ if $f'\in A^p_{p-1}$. Therefore the strength and the novelty of Theorem~\ref{thm:Hp-Ap-S} stems from the fact that, apart from being a much more general result, it also describes in concrete geometric terms, as an ``if and only if''-result, for which kind of weights such embeddings in fact happen. The last part of the theorem is a simple consequence of the Growth and Distortion theorems. Therefore, strictly speaking, Parts (i) and (ii) are not needed for our applications to conformal maps. However, we have taken as one of our goals on the study of conformal maps to underline in which steps the univalence is in essence and when the estimates in question are actually properties of all analytic functions and under which hypothesis on the weights and parameters involved. The proof of Theorem~\ref{thm:Hp-Ap-S} is given in Section~\ref{sec:thm4}.

Our discoveries so far are closely related to the Littlewood-Paley estimates. It was recently shown in \cite[Theorem~6]{PR2021} that
	\begin{equation}\label{PR2021-f'-Ap-charac}
	\|f'\|_{A^p_{\om_{[p]}}}+|f(0)|\lesssim\|f\|_{A^p_\om},\quad f\in\H(\D),
	\end{equation}
if and only if $\om\in\DD$. Meanwhile this neat characterization is certainly satisfactory as an ``if and only if''-result, by combining Theorems~\ref{theorem:norm} and \ref{thm:Hp-Ap-S} we obtain the following partial refinement in the case $2<p<\infty$.

\begin{corollary}
If $\om\in\DD$ and $2<p<\infty$ then
	\begin{equation}\label{PR2021-f'-Ap-charac-refcase1}
	\|f'\|_{A^p_{\om_{[p]}}}+|f(0)|
	\lesssim\|f'\|_{A^p_{\widetilde{\om}_{[p]}}}+|f(0)|
	\lesssim\|f\|_{A^p_\om},\quad f\in\H(\D).
	\end{equation}
\end{corollary}

To be precise, Theorem~\ref{theorem:norm} with $q=p+1$ shows that the first inequality in \eqref{PR2021-f'-Ap-charac-refcase1} is equivalent to $\om\in\DD$ for any $0<p<\infty$, and the second inequality is an easy consequence of Theorem~\ref{thm:Hp-Ap-S} applied to the weight $z\mapsto \int_{|z|}^1\om(t)t\,dt$ in place of $\om$. Observe that this weight belongs to~$\DD$ if and only if $\om$ does, as is seen by a simple integration by parts and using the moment condition $\om_x\lesssim\om_{2x}$ which characterizes the class $\DD$. It is easy to see that the refinement given in the corollary is in some cases a significant improvement to the inequality \eqref{PR2021-f'-Ap-charac}. Namely, by considering the weight $v_\alpha$ defined in \eqref{V} as an easy example, we see that in this case the integrands in the left-hand terms of \eqref{PR2021-f'-Ap-charac-refcase1} differ from each other by a logarithmic term, and hence the right-hand inequality gives a much better lower estimate for $\|f\|_{A^p_\om}$ in terms of $f'$ than what \eqref{PR2021-f'-Ap-charac} does. Therefore it does not come as a surprise that the refined estimate \eqref{PR2021-f'-Ap-charac-refcase1} can be efficiently applied in the operator theory. In a forthcoming work on Toeplitz operators it plays a crucial role in certain arguments~\cite{Toeplitz}.

We have one more quantity to be defined before representing our findings on conformal maps. For $0<p<\infty$ and a radial weight $\om$, we define
	\begin{equation}\label{Sclassdef}
	\begin{split}
	\|f\|_{S_{\om}^{p}}^p
	&=\int_0^1\left(\int_{D(0,r)}|f'|^2\,dA\right)^{\frac{p}{2}}\om(r)\,dr
	=\int_0^1\Area(f(D(0,r)))^{\frac{p}{2}}\om(r)\,dr,\quad f\in\H(\D),
	\end{split}
	\end{equation}
where $\Area(f(D(0,r)))$ denotes the area of image of $D(0,r)$ under $f$ counting multiplicities. Therefore the definition of $S^p_\om$ has obvious geometric flavor. The spaces $H^p_\om$ and $S^p_\om$ have their origin for the standard radial weights in \cite{HT1978,M-P1983}, and they appear in the study of conformal maps in \cite{GPP,HT1978,PG-R2008}.

Our first main result on conformal maps generalizes \eqref{H^p-1} and \eqref{Eq:thm-Bergman-2}, and reads as follows.

\begin{theorem}\label{Thm1}
Let $0<p<\infty$ and $\om\in\DDD$. Then
	$$
	\|f\|_{H_{\om}^{p}}^p+|f(0)|^p
	\asymp\|f\|_{S_{\om}^{p}}^p+|f(0)|^p
	\asymp J^p_{\om}(f),\quad f\in\U.
	$$
\end{theorem}

Theorem~\ref{Thm1} is proved in several individual steps and we actually prove much more than what is stated above. Our aim is to track meticulously the information equally on the necessity of injectivity of the function as well as on the smoothness and growth hypotheses on the weight involved in every step of the proof. In the case when $0<p<2$ and $\om\in\DDD$, we have the chain of inequalities
	\begin{equation}\label{A}
	\sum_{k=1}^\infty|\widehat{f}(k)|^pk^{p-1}\om_{2k}
	\lesssim\|f\|_{S_{\omega}^p}
	\lesssim\|f\|_{H_{\omega}^p}
	\lesssim\|f'\|_{A^p_{\widehat{\om}_{[p-2]}}},\quad f\in\H(\D),
	\end{equation}
which does not rely on the conformality of the function. The first two inequalities above are consequences of Lemmas~\ref{Lemma:H-versus-S} and \ref{Lemma:HL-versus-S} from Section~\ref{Sec:proofs of theorems on conf maps}, and they are in general false, even for the class $\SSS$, if $\om\in\DD\setminus\DDD$. Moreover, the right most estimate in \eqref{A} is satisfied for all $f\in\H(\D)$ if and only if $\widehat{\om}_{[p-2]}\in\DD$ by Theorem~\ref{thm:Hp-Ap-S}, provided $\widehat{\om}_{[p-2]}\in L^1$. To complete the picture in this direction we will also show in Proposition~\ref{lemma:HL-(p-2)} that $\sum_{k=1}^\infty|\widehat{f}(k)|^pk^{p-1}\om_{2k}\lesssim\|f'\|_{A^p_{\widehat{\om}_{[p-2]}}}$ for all $f\in\H(\D)$, provided $\widehat{\omega}_{[p-2]}\in\DD$. This case involving the extreme quantities in \eqref{A} is a pretty straightforward application of the Hardy-Littlewood inequality. Similar analysis is done on the case $2<p<\infty$, though the inequalities in \eqref{A} must be reversed.

To achieve the statement in Theorem~\ref{Thm1} we then must rely on the univalence. Lemma~\ref{lemma:H,S-J-univalent} states that
	\begin{equation}\label{B}
	\|f\|_{H_{\om}^{p}}^p+\|f\|_{S_{\om}^{p}}^p+|f(0)|^p\lesssim J^p_\om(f),\quad f\in\U,
	\end{equation}
for each fixed radial weight $\om$ and $0<p<\infty$. The injectivity is essential in this estimate because the asymptotic inequality is valid for all $f\in\H(\D)$ only if $\om$ vanishes almost everywhere, and that is certainly not allowed by our initial hypothesis on the strict positivity of the tail integral $\widehat{\om}$. Further, Lemma~\ref{P-lemma} gives
	\begin{equation}\label{C}
	J^p_\om(f)\lesssim\|f\|_{S_{\om}^{p}}^p+|f(0)|^p,\quad f\in\H(\D),
	\end{equation}
for each fixed $0<p<\infty$, provided $\om\in\DDD$. This estimate is in general false, even for the class $\SSS$, if $\om\in\DD\setminus\DDD$. By combining \eqref{A}, \eqref{B} and \eqref{C} we obtain the statement of Theorem~\ref{Thm1} in the case $0<p<2$.

In the case $2\le p<\infty$ we can relate the quantity $J^p_\om(f)$ to the norm of $f$ in the weighted Dirichlet space appearing in Theorem~\ref{thm:Hp-Ap-S}. Moreover, quantities involving other integral means are also related to the behavior of $J^p_\om(f)$. For this purpose, we define
	$$
	I_{p,q,\om}(f)=\int_0^1M_q^p(r,f')(1-r)^{p(1-\frac1q)}\om(r)\,dr
	$$
for $0<p,q<\infty$ and $f\in\H(\D)$.

\begin{theorem}\label{Thm2}
Let $2\le p,q<\infty$ and $\om\in\DDD$. Then
	\begin{equation}\label{thm:Dirichlet-asymp}
	\|f\|_{D^p_{\om_{[p-1]}}}^p+|f(0)|^p
	\asymp\|f\|_{D^p_{\widehat{\om}_{[p-2]}}}^p+|f(0)|^p
	\asymp J_{\om}^p(f)
	\asymp I_{p,q,\om}(f)+|f(0)|^p,\quad f\in\U.
	\end{equation}
\end{theorem}

As mentioned before, in the case $2<p<\infty$ and $\om\in\DDD$ we obtain the estimates
	$$
	\|f\|_{D^p_{\widehat{\om}_{[p-2]}}}
	\lesssim\|f\|_{H_{\omega}^p}
	\lesssim\|f\|_{S_{\omega}^p}
	\lesssim\sum_{k=1}^\infty|\widehat{f}(k)|^pk^{p-1}\om_k,\quad f\in\H(\D),
	$$
meanwhile Proposition~\ref{prop:Jp-Ap-om} and Theorem~\ref{thm-iff} yield
	$$
	J^p_\om(f)\lesssim\|f\|_{D^p_{\om_{[p-1]}}}^p+|f(0)|^p
	\asymp\|f\|_{D^p_{\widehat{\om}_{[p-2]}}}^p+|f(0)|^p,\quad f\in\H(\D).
	$$
These inequalities combined with \eqref{B} prove Theorem~\ref{Thm1} for $2<p<\infty$, and show that the first three quantities in \eqref{thm:Dirichlet-asymp} are comparable. The proof of the last asymptotic equality in \eqref{thm:Dirichlet-asymp} is basically independent of the arguments involved in the other steps of the proofs. The proof of Theorems~\ref{Thm1} and~\ref{Thm2} are presented in Section~\ref{Sec:proofs of theorems on conf maps}.

It was shown in \cite{Walsh2000} that if $1<p<\infty$ and $\Omega=f(\D)$, then
	\begin{equation}\label{eq:Walsh-Besov}
	\|f\|_{B^p}^p=\int_\D|f'(z)|^p(1-|z|^2)^{p-2}\,dA(z)\asymp \int_{\Omega}d(w,\partial\Omega)^{p-2}dA(w),\quad f\in\U,
	\end{equation}
where $d(w,\partial\Omega)$ denotes the Euclidean distance from $w$ to the boundary $\partial\Omega$. Theorems~\ref{Thm1} and~\ref{Thm2} combined have an interesting consequence related to this result concerning the Besov spaces $B^p$. Namely, they yield the geometric characterization
	$$
	\|f\|_{D^p_{\widehat{\om}_{[p-2]}}}^p
	=\int_\D|f'(z)|^p(1-|z|^2)^{p-2}\widehat{\om}(z)\,dA(z)\asymp\int_0^1\Area(f(D(0,r)))^{\frac{p}{2}}\om(r)\,dr,
	$$
provided $\om\in\DDD$ and $2\le p<\infty$. The space $D^p_{\widehat{\om}_{[p-2]}}$ is obviously bigger than $B^p$ due to the factor $\widehat{\om}$ involved in the weight. The quantity on the right hand side is certainly of different nature than that of \eqref{eq:Walsh-Besov}, yet both have their own obvious geometric interpretations. For more on univalent functions in Besov spaces, see \cite{DGV2002}.

The question of when $J_{\om}^p(f)\asymp I_{p,q,\om}(f)+|f(0)|^p$ for all $f\in\U$ remains in general unsettled. It is in fact a nontrivial task to answer this question even when $\om\equiv1$ which corresponds to the Hardy space case. For recent developments and relevant references, see \cite{PG-R-V2023}. An interested reader may treat some other parameter values than those appearing in Theorem~\ref{Thm2} by methods used in \cite{PG-R-V2023} and in its references, but we won't push things here to that direction mainly because we do not know how to make a substantial progress in the topic. We mention here only the fact that univalent functions in the Hardy space $H^p$ and the Dirichlet-type space $D^p_{p-1}$ are the same by the main result in \cite{B-G-P2004}. By combining Theorems~\ref{Thm1} and~\ref{Thm2} with $\om\equiv1$ we recover this result in the case $2<p<\infty$, though the most involved part of the proof in \cite{B-G-P2004} concerns the case $0<p<1$.

A careful reader observes another gap in our results. Namely, it is natural to ask for the role of the quantity $\sum_{k=1}^\infty|\widehat{f}(k)|^pk^{p-1}\om_k$ in this setting. Even if coefficient problems do not lie in the core of this work, we will give a partial answer. Namely, Lemma~\ref{Thm:coefficients} states that for $1\le p<\infty$ and $\om\in\M$ we have
	$$
	J^p_\om(f)\lesssim\sum_{k=0}^\infty|\widehat{f}(k)|^p(k+1)^{p-1}\om_k,\quad f\in\H(\D).
	$$
The converse of this inequality is valid for close-to-convex functions. Recall that $f\in\H(\D)$ is close-to-convex if there exists a convex function $g$ such that the real part of the quotient $f'/g'$ is strictly positive on $\D$,~see \cite[Chapter~2]{Duren83} and \cite[Chapter~2]{Pommerenke1} for information on this important subclass of univalent functions. These observations combined with \eqref{A} and \eqref{B} yield our last main result which reads as follows.

\begin{theorem}\label{thm:coefficients} Let $1\le p\le2$ and $\om\in\DDD$. Then
	\begin{equation}\label{eq:J-asymp-coeff}
	J_{\om}^{p}(f)
	\asymp\sum_{k=0}^\infty|\widehat{f}(k)|^p(k+1)^{p-1}\om_k,\quad f\in\U.
	\end{equation}
Moreover, this estimate is valid for all close-to-convex functions if $1\le p<\infty$.
\end{theorem}

The proof of Theorem~\ref{thm:coefficients} is given in Section~\ref{last section}. It is worth noticing that the case $\om\equiv1$ reduces to the known fact that the univalent functions in the Hardy space $H^p$ are the same as those in the Hardy-Littlewood space $\HL_p$ which consists of those $f\in\H(\D)$ whose Maclaurin coefficients satisfy $\sum_{k=0}^\infty|\widehat{f}(k)|^p(k+1)^{p-2}<\infty$~\cite{HT1978}.

Recently, a weighted version of $\HL_p$ naturally emerges in relation to the study of integration operators~\cite{PRW2023}. There, for $0<p<\infty$ and a radial weight $\om$, the weighted Hardy-Littlewood space $\HL^{\om}_{p}$ was defined by the condition
	$$
	\nm{f}^p_{\HL_p^{\om}}=\sum_{k=0}^{\infty}\abs{\widehat{f}(k)}^p (k+1)^{p-2}\om_{kp+1}<\infty.
	$$
For $\om\in\DD$ this quantity is comparable to the right hand side of \eqref{eq:J-asymp-coeff} because an application of Lemma~\ref{D-hat-lemma}(iv) shows that in this case, for each fixed $0<p<\infty$, the moments satisfy $\om_{kp+1}\asymp\om_k$ for all $k\in\N$.

\section{Proof of Theorem~\ref{theorem:norm}}\label{sec:thm1}

We begin with a lemma which contains useful characterizations of the class~$\DD$. For a proof of the fact that (i)--(iv) are equivalent, see \cite[Lemma~2.1]{Pelaez2016} and \cite{PR2021}. The last part of the lemma is new. It is worth noticing that it fails in the case $\beta=0$ because by \cite[Theorem~3]{PR2020} there exists a weight $\om\not\in\DD$ such that $\widehat{\om}_{[-1]}\in\DD$. However, the proof shows that if $\om\in\DD$ then $\widehat{\om}_{[\beta-1]}\in\DD$ for each $-1<\beta<\infty$, provided $\widehat{\om}_{[\beta-1]}\in L^1$.

\begin{letterlemma}\label{D-hat-lemma}
Let $\om$ be a radial weight. Then the following statements are equivalent:
\begin{itemize}
\item[\rm(i)] $\om\in\DD$;
\item[\rm(ii)] There exist $C=C(\om)\ge 1$ and $\b=\b(\om)>0$ such that
    \begin{equation*}
    \begin{split}
    \frac{\widehat{\om}(r)}{(1-r)^\beta}\le C\frac{\widehat{\om}(t)}{(1-t)^\beta},\quad 0\le r\le t<1;
    \end{split}
    \end{equation*}
\item[\rm(iii)] For some (equivalently for each) $\b>0$ there exists a constant $C=C(\om,\b)>0$ such that
	\begin{equation}\label{eq:D-hat-lemma-iii}
	x^{\b}\left(\om_{[\b]}\right)_x\le C \om_x, \quad 0\le x < \infty;
	\end{equation}
\item[\rm(iv)] There exists $C=C(\om)>0$ such that $\om_{x}\le C\om_{2x}$ for all $0\le x<\infty$;
\item[\rm(v)] $\widehat{\om}_{[\beta-1]}\in\DD$ for some (equivalently for each) $\b>0$.
\end{itemize}
\end{letterlemma}

\begin{proof}
Since (i)--(iv) are equivalent by \cite[Lemma~2.1]{Pelaez2016} and \cite{PR2021}, it remains to show the equivalence between (i) and (v). Assume first $\om\in\DD$. Then
	\begin{equation*}
	\begin{split}
	\int_r^1\widehat{\om}_{[\beta-1]}(t)\,dt
	\lesssim\int_r^1\widehat{\om}\left(\frac{1+t}{2}\right)(1-t)^{\beta-1}\,dt
	=2^{\beta}\int_{\frac{1+r}{2}}^1\widehat{\om}_{[\beta-1]}(s)\,ds,\quad 0\le r<1,
	\end{split}
	\end{equation*}
and hence $\widehat{\om}_{[\beta-1]}\in\DD$. This part of the proof is valid for each $-1<\beta<\infty$, provided $\widehat{\om}_{[\beta-1]}\in L^1$.

Conversely, if $\widehat{\om}_{[\beta-1]}\in\DD$, then
	\begin{equation*}
	\begin{split}
	\widehat{\om}\left(\frac{1+r}{2}\right)\frac{(1-r)^{\beta}}{\beta}\left(1-\frac1{2^{\beta}}\right)
	&\le\int_r^1\widehat{\om}_{[\beta-1]}(t)\,dt
	\lesssim\int_{\frac{3+r}{4}}^1\widehat{\om}_{[\beta-1]}(t)\,dt\\
	&\le\widehat{\om}\left(\frac{3+r}{4}\right)\frac{1}{\beta 4^{\beta}}(1-r)^{\beta},\quad 0\le r<1,
	\end{split}
	\end{equation*}
and hence $\widehat{\om}(r)\lesssim\widehat{\om}\left(\frac{1+r}{2}\right)$ for $\frac12\le r<1$. It follows that $\om\in\DD$.
\end{proof}

The next elementary lemma is very useful for our purposes. The special case $q=p$ shows that  $\widehat{\om}\lesssim\widehat{\nu}$ on $[0,1)$ is a sufficient condition for the identity operator $I:A^p_\nu\to A^p_\om$ to be bounded. For Carleson embedding theorems for the weighted Bergman space $A^p_\om$, see \cite{Pelaez2016,P-R2015,PelRatEmb,P-R2014} in the case $\om\in\DD$, and \cite{LR,LRW} for $\om\in\DDD$.

\begin{lemma}\label{AuxLemmaEmbedding}
Let $0<p<\infty$, $0<q\le\infty$ and $0\le\rho<1$, and let $\om$ and $\nu$ be radial weights such that $\widehat{\om}\lesssim\widehat{\nu}$ on $[\rho,1)$. Then
	$$
	\int_\rho^1M_q^p(r,f)\om(r)\,dr\lesssim\int_\rho^1M_q^p(r,f)\nu(r)\,dr,\quad f\in\H(\D).
	$$
\end{lemma}

\begin{proof}
If $\int_0^1M_q^p(r,f)\nu(r)\,dr<\infty$, then the hypothesis $\widehat{\om}\lesssim\widehat{\nu}$ yields
    $$
    M^p_q(r,f)\widehat{\om}(r)
		\lesssim M^p_q(r,f)\widehat{\nu}(r)
		\le\int_r^1M^p_q(t,f)\nu(t)\,dt\to0,\quad r\to1^-.
    $$
Hence, an integration by parts together with another application of the hypothesis $\widehat{\om}\lesssim\widehat{\nu}$ gives
  \begin{equation}\label{eq:intbyparts}
  \begin{split}
  \int_{\rho}^{1}M_q^p(r,f)\om(r)\,dr
  &=M_q^p(\rho,f)\widehat{\om}(\rho)+\int_{\rho}^{1}\frac{\partial}{\partial r}M_q^p(r,f)\widehat{\om}(r)\,dr\\
  &\lesssim M_q^p(\rho,f)\widehat{\nu}(\rho)+\int_{\rho}^{1} \frac{\partial}{\partial r}M_q^p(r,f)\widehat{\nu}(r)\,dr\\
  &=\int_{\rho}^{1}M_q^p(r,f)\nu(r)\,dr,
  \end{split}
  \end{equation}
which proves the assertion.
\end{proof}

With these preparations we are ready to prove Theorem~\ref{theorem:norm}.

\medskip

\noindent{\emph{Proof of Theorem~\ref{theorem:norm}.}} First observe that the statement is valid for $q=1$ by \cite[Theorem~8]{PR2020}. Let $q\in(0,\infty)\setminus\{1\}$. If $\widehat{\om}_{[q-2]}\not\in L^1$, then \eqref{Eq:norm-Bergman} is obvious. Assume now $\widehat{\om}_{[q-2]}\in\DD$. Then
	\begin{equation*}
	\begin{split}
	\int_{\frac{1+r}{2}}^1\widehat{\om}_{[q-2]}(t)\,dt
	&\gtrsim\int_r^1\widehat{\om}_{[q-2]}(t)\,dt
	\ge\int_r^{\frac{1+r}{2}}\widehat{\om}_{[q-2]}(t)\,dt
	\\&\geq\widehat{\om}\left(\frac{1+r}{2}\right)\left(1-\frac{1+r}{2}\right)^{q-1}\frac{1-2^{q-1}}{1-q},\quad 0\le r<1,
	\end{split}
	\end{equation*}
and it follows that
	\begin{equation}\label{eq:omhat-q-int-compare}
	\widehat{\om}_{[q-1]}(r)\lesssim\int_r^1\widehat{\om}_{[q-2]}(t)\,dt,\quad 0\le r<1.
	\end{equation}
Then Fubini's theorem yields
	\begin{equation}\label{eq:thm1-fubini-calc1}
	\begin{split}
	\int_r^1\widehat{\om}_{[q-2]}(t)\,dt
	&=\int_r^1\om(s)\left(\int_r^s(1-t)^{q-2}\,dt\right)ds\\
	&=\frac{\widehat{\om}_{[q-1]}(r)}{q-1}-\frac{1}{q-1}\int_r^1\om_{[q-1]}(s)\,ds,\quad 0\le r<1.
	\end{split}
	\end{equation}
This combined with \eqref{eq:omhat-q-int-compare} gives
	\begin{equation}\label{eq:thm1-combine-calc2}
	\int_r^1\om_{[q-1]}(s)\,ds\lesssim\int_r^1\widehat{\om}_{[q-2]}(t)\,dt,\quad 0\le r<1,
	\end{equation}
which in turn guarantees \eqref{Eq:norm-Bergman} by Lemma~\ref{AuxLemmaEmbedding} with $q=p$.

Conversely, assume that \eqref{Eq:norm-Bergman} is satisfied and $\widehat{\om}_{[q-2]}\in L^1$. Consider the monomial $m_n(z)=z^n$, where $n\in\N\cup\{0\}$. An integration by parts and \eqref{eq:thm1-fubini-calc1} give
	\begin{equation*}
	\begin{split}
	\frac1{2\pi}\|m_n\|_{A^p_{\om_{[q-1]}}}^p
	&=\left(\om_{[q-1]}\right)_{np+1}
	=(np+1)\int_0^1r^{np}\left(\int_r^1\om_{[q-1]}(t)\,dt\right)dr\\
	&=(np+1)\left(\widehat{\om}_{[q-1]}\right)_{np}
	-(q-1)(np+1)\int_0^1r^{np}\left(\int_r^1\widehat{\om}_{[q-2]}(t)\,dt\right)dr.
	\end{split}
	\end{equation*}
In the case $0<q<1$ this together with \eqref{Eq:norm-Bergman} yields
	\begin{equation*}
	\begin{split}
	\left(\widehat{\om}_{[q-2]}\right)_{np+1}
	=\frac1{2\pi}\|m_n\|_{A^p_{\widehat{\om}_{[q-2]}}}^p
	\gtrsim(np+1)\left(\left(\widehat{\om}_{[q-2]}\right)_{[1]}\right)_{np},\quad n\in\N.
	\end{split}
	\end{equation*}
Let $p\le x<\infty$, and choose $n\in\N$ such that $np\le x\le (n+1)p$. Then
	\begin{equation}\label{eq:ntox-calc1}
	\begin{split}
	\left(\widehat{\om}_{[q-2]}\right)_{x}
	&\ge\left(\widehat{\om}_{[q-2]}\right)_{(n+1)p}
	\asymp\left(\widehat{\om}_{[q-2]}\right)_{np+1}
	\gtrsim(np+1)\left(\left(\widehat{\om}_{[q-2]}\right)_{[1]}\right)_{np}\\
	&\ge(x-p+1)\left(\left(\widehat{\om}_{[q-2]}\right)_{[1]}\right)_{x}
	\gtrsim x\left(\left(\widehat{\om}_{[q-2]}\right)_{[1]}\right)_{x},\quad p\le x<\infty,
	\end{split}
	\end{equation}
and Lemma~\ref{D-hat-lemma}(iii) implies $\widehat{\om}_{[q-2]}\in\DD$. If $1<q<\infty$, then an integration by parts gives
	\begin{equation}\label{eq:om-moment-id1}
	\begin{split}
	\left(\om_{[q-1]}\right)_{np+1}
	=(np+1)\left(\widehat{\om}_{[q-1]}\right)_{np}-(q-1)\left(\widehat{\om}_{[q-2]}\right)_{np+1},\quad n\in\N\cup\{0\}.
	\end{split}
	\end{equation}
This and \eqref{Eq:norm-Bergman} yield
\begin{equation*}
	\begin{split}
	\left(\widehat{\om}_{[q-2]}\right)_{np+1}
	\gtrsim(np+1)\left(\left(\widehat{\om}_{[q-2]}\right)_{[1]}\right)_{np},\quad n\in\N\cup\{0\},
	\end{split}
	\end{equation*}
and again we deduce $\widehat{\om}_{[q-2]}\in\DD$ by Lemma~\ref{D-hat-lemma}(iii). Observe that in this case $\widehat{\om}_{[q-2]}\in\DD$ is equivalent to $\om\in\DD$ by Lemma~\ref{D-hat-lemma}(v) with $\beta=q-1>0$.\hfill$\Box$

\section{Proof of Theorem~\ref{thm-iff}}\label{sec:thm2}

A basic result that we will need is a set of characterizations of weights in~$\Dd$ given in the next lemma. The characterization (ii) is well known and a detailed proof can be found in \cite{RosaPelaez2022}, while the characterizations (iii)-(vii) are unpublished results by J. A. Pel\'aez and the second author, and (viii) is new. In this work we do not use the conditions (iii) and (iv) as such but as the proof passes naturally through them and the characterizations seem useful they are included here for the convenience of the reader and for further reference. The points $\r_n=\r_n(\om,K)\in[0,1)$ appearing in (iv) are defined by the identity
	$$
	\r_n=\r_n(\om,K)=\min\{0\le r<1:\widehat{\om}(r)=\widehat{\om}(0)K^{-n}\},\quad 1<K<\infty,\quad n\in\N\cup\{0\}.
	$$
Observe that $\r_0=0<\r_1<\cdots<\r_n<\r_{n+1}<\cdots$ for all $n\in\N$, and $\r_n\to1^-$ as $n\to\infty$.

\begin{lemma}\label{Lemma:weights-in-R}
Let $\om$ be a radial weight. Then the following statements are
equivalent:
\begin{itemize}
\item[\rm(i)] $\om\in\Dd$;
\item[\rm(ii)] There exist $C=C(\om)>0$ and $\b=\b(\om)>0$ such that
    \begin{equation*}
    \begin{split}
    \frac{\widehat{\om}(t)}{(1-t)^\beta}\le C\frac{\widehat{\om}(r)}{(1-r)^\beta},\quad 0\le r\le t<1;
    \end{split}
    \end{equation*}
\item[\rm(iii)] For some (equivalently for each) $\g\in(0,\infty)$, there exists $C=C(\g,\om)>0$ such that
    $$
    \int_0^r\frac{dt}{\widehat{\om}(t)^\g(1-t)}\le\frac{C}{\widehat{\om}(r)^\g},\quad 0\le r<1;
    $$
\item[\rm(iv)] For some (equivalently for each) $K>1$, there exists $C=C(\om,K)>0$ such that $1-\r_n\le C(1-\r_{n+1})$ for all $n\in\N\cup\{0\}$;
\item[\rm(v)] For some (equivalently for each) $\b\in(0,\infty)$, there exists $C=C(\b,\om)>1$ such that
    $$
    \widehat{\om}(r)\le C\frac{\widehat{\om_{[\b]}}(r)}{(1-r)^\b},\quad 0\le r<1;
    $$
\item[\rm(vi)] For some (equivalently for each) $\b\in(0,\infty)$, there exists $C=C(\b,\om)\in(0,1)$ such that
    $$
    \frac{\int_r^1\widehat{\om}(t)\b(1-t)^{\b-1}\,dt}{(1-r)^{\b}}\le C\widehat{\om}(r),\quad 0\le r<1;
    $$
\item[\rm(vii)] For some (equivalently for each) $\g\in(0,\infty)$, there exists $C=C(\g,\om)>0$ such that
    $$
    \int_r^1 \frac{\widehat{\om}(s)^\g}{1-s}\,ds\le C\widehat{\om}(r)^\g,\quad 0\le r<1;
    $$
\item[\rm(viii)] $\widehat{\om}_{[\beta-1]}\in\Dd$ for some (equivalently for each) $0\le\beta<\infty$.
\end{itemize}
\end{lemma}

\begin{proof}
By \cite[Lemma~B]{RosaPelaez2022} we know that $\om\in\Dd$ if and only if (ii) is satisfied.

Assume (ii) and let $0<r<1$. Then
    \begin{equation*}
    \begin{split}
    \int_0^r\frac{dt}{\widehat{\om}(t)^\g(1-t)}\le\frac{C^\g(1-r)^{\g\b}}{\widehat{\om}(r)^\g}\int_0^r\frac{dt}{(1-t)^{\g\b+1}}
    =\frac{C^\g}{\g\b\widehat{\om}(r)^\g}\left(1-(1-r)^{\g\b}\right),
    \end{split}
    \end{equation*}
and (iii) follows.

If (iii) is satisfied and $0\le r\le t<1$, then
    \begin{equation*}
    \begin{split}
    \frac{1}{\widehat{\om}(r)^\g}\log\frac{1-r}{1-t}
    =\frac{1}{\widehat{\om}(r)^\g}\int_r^t\frac{ds}{1-s}
    \le\int_r^t\frac{ds}{\widehat{\om}(s)^\g(1-s)}
    \le \frac{C}{\widehat{\om}(t)^\g},
    \end{split}
    \end{equation*}
where $C=C(\g,\om)>0$. By setting $t=1-\frac{1-r}{K}$, where $K>1$, we deduce
    $$
    \widehat{\om}(r)\ge\left(\frac{\log K}{C}\right)^\frac1{\gamma}\widehat{\om}\left(1-\frac{1-r}{K}\right),\quad 0\le r<1,
    $$
from which (i) follows by choosing $K$ sufficiently large. Thus (i)--(iii) are equivalent.

We will show next that (iv)--(vi) are equivalent to the first three conditions. To see this, note first that, by choosing $t=\r_{n+1}$ and $r=\r_n$ in (ii), we deduce (iv).

Assume (iv), and let $0<\b<\infty$ and $0<r<1$. Let $K>1$ and choose $n\in\N\cup\{0\}$
such that $\r_n\le r<\r_{n+1}$. Then
    \begin{equation*}
    \begin{split}
    \frac{\widehat{\om_{[\b]}}(r)}{(1-r)^\b}
    \ge\frac{\widehat{\om_{[\b]}}(\r_{n+1})}{(1-\r_n)^\b}
    &\ge\frac1{(1-\r_n)^\b}\sum_{j=n+1}^\infty(1-\r_{j+1})^\b\int_{\r_j}^{\r_{j+1}}\om(t)\,dt\\
    &=\frac{K-1}{K^2}\sum_{j=n+1}^\infty\left(\frac{1-\r_{j+1}}{1-\r_n}\right)^\b\frac{\widehat{\om}(0)}{K^{j-1}}\\
    &\ge\frac{K-1}{K^2}\left(\frac{1-\r_{n+2}}{1-\r_n}\right)^\b\widehat{\om}(\r_n)
    \ge\frac{K-1}{K^2C^{2\b}}\widehat{\om}(r),
    \end{split}
    \end{equation*}
and the inequality in (v) follows.

An integration by parts shows that
    $$
    \widehat{\om_{[\b]}}(r)=\int_r^1\om(t)(1-t)^\b\,dt=\widehat{\om}(r)(1-r)^\b-\int_r^1\widehat{\om}(t)\b(1-t)^{\b-1}\,dt,
    $$
and hence, for a given $\b>0$, the inequality in (v) is equivalent to
    \begin{equation*}
		\begin{split}
    \widehat{\om}(r)\le C\widehat{\om}(r)-\frac{C\int_r^1\widehat{\om}(t)\b(1-t)^{\b-1}\,dt}{(1-r)^\b}
    \quad\Longleftrightarrow\quad
		\frac{\int_r^1\widehat{\om}(t)\b(1-t)^{\b-1}\,dt}{(1-r)^\b}\le\frac{C-1}{C}\widehat{\om}(r).
		\end{split}
		\end{equation*}
We deduce that, for a given $\b>0$, the inequality in (v) implies that of~(vi).

Assume the inequality in (vi) for constants $\b>0$ and $C=C(\b,\om)\in(0,1)$. Choose $K$ sufficiently large
such that $1-\frac{1}{K^\b}>C$. Then
    \begin{equation*}
    \begin{split}
    C\widehat{\om}(r)
    \ge\frac{\int_r^{1-\frac{1-r}{K}}\widehat{\om}(t)\b(1-t)^{\b-1}\,dt}{(1-r)^\b}
    \ge\widehat{\om}\left(1-\frac{1-r}{K}\right)\left(1-\frac{1}{K^\b}\right),
    \end{split}
    \end{equation*}
and hence (i) is satisfied. By combining the steps above, we deduce that (i)--(vi) are equivalent.

We next show that (i) and (vii) are equivalent. Assume first (vii). Then, for each $K>1$, we have
    $$
    C\widehat{\om}(r)^\gamma
		\ge\int_r^1\frac{\widehat{\om}(s)^\g}{1-s}\,ds
		\ge\int_r^{1-\frac{1-r}{K}}\frac{\widehat{\om}(s)^\g}{1-s}\,ds
		\ge\widehat{\om}\left(1-\frac{1-r}{K} \right)^\g\log K,
    $$
and hence, by choosing $K$ sufficiently large such that $\log K>C$ we obtain (i). Conversely, assume (i) and let $\gamma>0$. Then
    $$
    \int_r^1\frac{\widehat{\om}(s)^\g}{1-s}\,ds
    \ge C^\g\int_r^1\frac{\widehat{\om}\left(1-\frac{1-s}{K}\right)^\g}{1-s}\,ds
    =C^\g\int_{1-\frac{1-r}{K}}^1\frac{\widehat{\om}(s)^\g}{1-s}\,ds,
    $$
and hence
    \begin{equation*}
    \begin{split}
    \int_r^1\frac{\widehat{\om}(s)^\g}{1-s}\,ds
    &\le\int_r^{1-\frac{1-r}{K}}\frac{\widehat{\om}(s)^\g}{1-s}\,ds
		+\frac{1}{C^\g}\int_r^1\frac{\widehat{\om}(s)^\g}{1-s}\,ds\\
    &\le\widehat{\om}(r)^\g\log K+\frac{1}{C^\g}\int_r^1\frac{\widehat{\om}(s)^\g}{1-s}\,ds,
    \end{split}
    \end{equation*}
that is,
    $$
    \int_r^1\frac{\widehat{\om}(s)^\g}{1-s}\,ds
		\le\widehat{\om}(r)^\g \frac{C^\g}{C^\g-1}\log K.
    $$
Thus (vii) is satisfied.

It remains to show that (viii) is equivalent to the other conditions. Assume first $\om\in\Dd$. Then there exist constants $C=C(\om)>1$ and $K=K(\om)>1$ such that
	\begin{equation*}
	\begin{split}
	\int_r^1\widehat{\om}_{[\beta-1]}(t)\,dt
	\ge C\int_r^1\widehat{\om}\left(1-\frac{1-t}{K}\right)(1-t)^{\beta-1}\,dt
	=CK^{\beta}\int_{1-\frac{1-r}{K}}^1\widehat{\om}_{[\beta-1]}(s)\,ds,\quad 0\le r<1,
	\end{split}
	\end{equation*}
and hence $\widehat{\om}_{[\beta-1]}\in\Dd$. This argument is valid for each $-1<\beta<\infty$, provided $\widehat{\om}_{[\beta-1]}\in L^1$.

Conversely, if $\widehat{\om}_{[\beta-1]}\in\Dd$, then (v) with $\beta=1$ yields
	$$
	\int_r^1\widehat{\om}_{[\beta-1]}(t)\,dt\lesssim\frac{\int_r^1\widehat{\om}_{[\beta]}(t)\,dt}{1-r}
	\le\widehat{\om}(r)\frac{(1-r)^{\beta}}{\beta+1},\quad 0\le r<1,
	$$
for each $-1<\beta<\infty$. If $0<\beta<\infty$, then $\om\in\Dd$ by (vi) while if $\beta=0$, then $\om\in\Dd$ by (vii).
\end{proof}

The next lemma concerns the class $\M$, and its proof can be found in \cite{PR2021}.

\begin{letterlemma}\label{M-lemma-1}
Let $\om$ be a radial weight. Then the following statements are equivalent:
\begin{itemize}
\item[\rm(i)] $\om\in\M$;
\item[\rm(ii)] There exist $C=C(\om)>0$ and $K=K(\om)>1$ such that
    $$
    \widehat{\om}(t)\le C\int_0^ts^{\frac1{K(1-t)}}\om(s)\,ds,\quad 1-\frac1K\le t<1;
    $$
\item[\rm(iii)] For some (equivalently for each) $\b>0$, there exists $C=C(\om,\b)>0$ such that
    $$
    \om_x\le Cx^\b\left(\om_{[\b]}\right)_x,\quad 1\le x<\infty.
    $$
\end{itemize}
\end{letterlemma}

Lemma~\ref{M-lemma-1}(ii) shows that $\M$ is closed under multiplication under any non-increasing weight. In particular, if $\om\in\M$, then $\om_{[\beta]}\in\M$ for each $\beta>0$.

The last result towards Theorem~\ref{thm-iff} before the proof itself is the next proposition which concerns the asymptotic inequality $\|f\|_{A^p_{\widehat{\om}_{[q-2]}}}\lesssim\|f\|_{A^p_{\om_{[q-1]}}}$. As mentioned in the introduction, we have been unable to judge if, in the case $1\le q<\infty$, the containment $\om\in\M$ is a sufficient condition for this estimate to hold for all $f\in\H(\D)$ unless $p$ is even. However, our impression is that the answer to this question should be affirmative.

\begin{proposition}\label{proposition-norm}
Let $\om$ be a radial weight and $0<p<\infty$. Then the following statements hold:
\begin{itemize}
\item[\rm(i)] If $0<q<1$, then
	\begin{equation}\label{norm-inequality}
	\|f\|_{A^p_{\widehat{\om}_{[q-2]}}}\lesssim\|f\|_{A^p_{\om_{[q-1]}}},\quad f\in\H(\D);
	\end{equation}
\item[\rm(ii)] If $1\le q<\infty$, then a necessary condition for \eqref{norm-inequality} to hold is that $\widehat{\om}_{[q-2]}$ belongs to $\M$;
\item[\rm(iii)] If $1<q<\infty$, then a necessary condition for \eqref{norm-inequality} to hold is that $\om$ belongs to $\M$;
\item[\rm(iv)] If $1\le q<\infty$, then a sufficient condition for \eqref{norm-inequality} to hold is $\om\in\Dd$.
\end{itemize}
\end{proposition}

\begin{proof}
(i) If $0<q<1$, then the asymptotic inequality \eqref{norm-inequality} is valid by \eqref{eq:thm1-fubini-calc1} and Lemma~\ref{AuxLemmaEmbedding}.

(ii) By testing \eqref{norm-inequality} with the monomial $m_n$ we obtain $\left(\widehat{\om}_{[q-2]}\right)_{np+1}\lesssim\left(\om_{[q-1]}\right)_{np+1}$ for all $n\in\N\cup\{0\}$. By applying \eqref{eq:om-moment-id1} on the right and re-organizing terms we deduce $\left(\widehat{\om}_{[q-2]}\right)_{np+1}\lesssim(np+1)\left(\left(\widehat{\om}_{[q-2]}\right)_{[1]}\right)_{np}$ for all $n\in\N\cup\{0\}$. By arguing as in \eqref{eq:ntox-calc1} this implies $\widehat{\om}_{[q-2]}\in\M$ by Lemma~\ref{M-lemma-1}(iii).

(iii) Fubini's theorem shows that
	$$
	\left(\widehat{\om}_{[q-2]}\right)_{x}
	=\int_0^1\om(r)\left(\int_0^r(1-t)^{q-2}t^x\,dt\right)dr,\quad 1\le x<\infty.
	$$
By considering the derivative of the function
	$$
	\psi(r)=\int_0^r(1-t)^{q-2}t^x\,dt-C\frac{r^{x+q}}{x^{q-1}},\quad \psi(0)=0,
	$$
and observing that
	$$
	\psi(1)=\int_0^1(1-t)^{q-2}t^x\,dt-\frac{C}{x^{q-1}}
	\ge\int_{1-\frac1x}^1(1-t)^{q-2}t^x\,dt-\frac{C}{x^{q-1}}
	\gtrsim\frac1{x^{q-1}},\quad 1\le x<\infty,
	$$
for $C=C(q)>0$ sufficiently small, we deduce $\psi(r)\ge0$ for all $0\le r<1$. It follows that
	$$
	\left(\widehat{\om}_{[q-2]}\right)_{x}\gtrsim x^{1-q}\om_{x+q}\asymp x^{1-q}\om_{x},\quad 1\le x<\infty.
	$$
This together with standard estimates and the inequality $\left(\widehat{\om}_{[q-2]}\right)_{np+1}\lesssim\left(\om_{[q-1]}\right)_{np+1}$, valid for all $n\in\N\cup\{0\}$, yields $\om_x\lesssim x^{q-1}\left(\om_{[q-1]}\right)_{x}$ for all $1\le x<\infty$. Therefore $\om\in\M$ by Lemma~\ref{M-lemma-1}(iii).

(iv) Let first $1<q<\infty$. By \eqref{eq:thm1-fubini-calc1} and Lemma~\ref{AuxLemmaEmbedding} it suffices to show that
	$$
	\widehat{\om}_{[q-1]}(r)\lesssim\int_r^1\om_{[q-1]}(t)\,dt,\quad 0\le r<1.
	$$
But this is equivalent to $\om\in\Dd$ by Lemma~\ref{Lemma:weights-in-R}(v) with $\beta=q-1$. If $q=1$, then by Lemma~\ref{Lemma:weights-in-R}(ii) there exists a constant $\beta=\beta(\om)>0$ such that
	$$
	\int_r^1\frac{\widehat{\om}(t)}{1-t}\,dt
	\lesssim\frac{\widehat{\om}(r)}{(1-r)^\beta}\int_r^1\frac{dt}{(1-t)^{1-\beta}}\,dt
	=\frac{\widehat{\om}(r)}{\beta},
	$$
and the assertion in the case $q=1$ follows by Lemma~\ref{Lemma:weights-in-R}(vii) with $\gamma=1$.
\end{proof}

We can now pull the proof of Theorem~\ref{thm-iff} together.

\medskip

\noindent{\emph{Proof of Theorem~\ref{thm-iff}.}}
(i) is an immediate consequence of Theorem~\ref{theorem:norm} and Proposition~\ref{proposition-norm}(i).


(ii) First observe that, for $1\le q<\infty$, $\om\in\DDD$ if and only if $\widehat{\om}_{[q-2]}\in\DDD$ by Lemmas~\ref{D-hat-lemma}(v) and~\ref{Lemma:weights-in-R}(viii) and \cite[Theorem~9]{PR2020}. Therefore the assertion follows by
Theorem~\ref{theorem:norm} and {\color{red}Proposition~\ref{proposition-norm}(ii)(iv)} and the identity $\DD\cap\M=\DDD=\DD\cap\Dd$. \hfill$\Box$

\section{Proof of Theorem~\ref{thm:A2-weight-M}}\label{sec:thm3}

The next lemma is an unpublished result by J. A. Pel\'aez and the second author. It contains a set of characterizations of the class $\M$, and it should be compared with Lemma~\ref{M-lemma-1} in Section~\ref{sec:thm2}.

\begin{lemma}\label{M-lemma}
Let $\om$ be a radial weight. Then the following statements are equivalent:
\begin{itemize}
\item[\rm(i)] $\om\in\M$;
\item[\rm(ii)] For some (equivalently for each) $0<\g<\infty$, there exists $C=C(\om,\gamma)>0$ such that
    $$
    \int_x^\infty\om_y^\g\frac{dy}{y}\le C\om_x^\gamma,\quad 1\le x<\infty;
    $$
\item[\rm(iii)] For some (equivalently for each) $0\le\b<\infty$, there exists $C=C(\om,\beta)>0$ such that
    $$
    \int_x^\infty\om_y\frac{dy}{y^{\b+1}}\le C\frac{\om_x}{x^\b},\quad 1\le x<\infty;
    $$
\item[\rm(iv)] For some (equivalently for each) $0\le\b<\infty$, there exists $C=C(\om,\b)>0$ such that
    $$
    \int_{1-\frac1x}^1\widehat{\om}_{[\beta-1]}(t)\,dt\le C\frac{\om_x}{x^\b},\quad 1\le x<\infty.
    $$
\end{itemize}
\end{lemma}

\begin{proof}
Assume $\om\in\M$. Then
    \begin{equation*}
    \begin{split}
    \int_x^\infty\om_y^\gamma\frac{dy}{y}
    &=\int_x^{Kx}\om_y^\gamma\frac{dy}{y}
    +\int_{Kx}^\infty\om_y^\gamma\frac{dy}{y}
    \le\om_x^\gamma\log K
    +\int_x^\infty\om_{Kt}^\gamma\frac{dt}{t}\\
    &\le\om_x^\gamma\log K
    +\frac1{C^\gamma}\int_x^\infty\om_{t}^\gamma\frac{dt}{t},
    \end{split}
    \end{equation*}
and since $C>1$, (ii) follows by re-organizing terms. Conversely, (ii) implies
    \begin{equation*}
    C\om_x^\gamma\ge\int_x^\infty\om_y^\gamma\frac{dy}{y}\ge\int_x^{Kx}\om_y^\gamma\frac{dy}{y}\ge\om_{Kx}^\gamma\log K,
    \end{equation*}
and hence, by choosing $K=K(\om)$ sufficiently large so that $\log K>C$, we deduce (i). Thus (i) and (ii) are equivalent.

Next, observe that (ii), with $\g=1$, is equivalent to (iii) with $\b=0$, so it suffices to show that (iii) and (iv) are equivalent.

Let $1<x<\infty$. Then
    \begin{equation*}
    \begin{split}
    \int_x^\infty\om_y\frac{dy}{y^{1+\beta}}
    &=\int_x^\infty\left(\int_0^{1-\frac1y}r^y\om(r)\,dr\right)\frac{dy}{y^{1+\beta}}
    +\int_x^\infty\left(\int_{1-\frac1y}^1r^y\om(r)\,dr\right)\frac{dy}{y^{1+\beta}}
    =I_1(x)+I_2(x),
    \end{split}
    \end{equation*}
where
    \begin{equation*}
    \begin{split}
    I_1(x)&=\int_{x}^\infty\left(\int_0^{1-\frac1x}r^y\om(r)\,dr\right)\frac{dy}{y^{1+\beta}}
    +\int_{x}^\infty\left(\int_{1-\frac1x}^{1-\frac1y}r^y\om(r)\,dr\right)\frac{dy}{y^{1+\beta}}
    =I_{11}(x)+I_{12}(x)
    \end{split}
    \end{equation*}
and, by the change of variable $y=\frac1{1-t}$,
    \begin{equation}\label{eq:I2-calc1}
    \begin{split}
    I_2(x)
		&=\int_x^\infty\left(\int_{1-\frac1y}^1r^y\om(r)\,dr\right)\frac{dy}{y^{1+\beta}}
		=\int_{1-\frac1x}^1\left(\int_t^1r^\frac1{1-t}\om(r)\,dr\right)\frac{dt}{(1-t)^{1-\beta}}\\
		&\asymp\int_{1-\frac1x}^1\left(\int_t^1\om(r)\,dr\right)\frac{dt}{(1-t)^{1-\beta}}=\int_{1-\frac1x}^1\widehat{\om}_{[\beta-1]}(t)\,dt.
    \end{split}
    \end{equation}
Since, by Fubini's theorem,
    \begin{equation*}
    \begin{split}
    I_{11}(x)
    &=\int_0^{1-\frac1x}\om(r)\left(\int_{x}^\infty r^y\frac{dy}{y^{\beta+1}}\right)dr
    \le\int_0^{1-\frac1x}\frac{\om(r)}{\log\frac1r}\left(\int_{x}^\infty r^y\log\frac1r\,dy\right)\frac{dr}{x^{\beta+1}}\\
    &=\int_0^{1-\frac1x}\frac{\om(r)}{\log\frac1r}r^x\frac{dr}{x^{\beta+1}}
    \le\left(x^{\beta+1}\log\frac1{1-\frac1x}\right)^{-1}\int_0^{1-\frac1x}r^x\om(r)\,dr\lesssim\frac{\om_x}{x^{\beta}}
    \end{split}
    \end{equation*}
and
    \begin{equation*}
    \begin{split}
    I_{12}(x)
    &=\int_{1-\frac1x}^1\om(r)\left(\int_{\frac1{1-r}}^\infty r^y\frac{dy}{y^{\beta+1}}\right)dr
    \le\int_{1-\frac1x}^1\om(r)\left(\int_{\frac1{1-r}}^\infty r^y(1-r)^{\beta+1}\,dy\right)dr\\
    &\le\int_{1-\frac1x}^1\om(r)(1-r)^\beta\left(\int_{\frac1{1-r}}^\infty r^y\log\frac1r\,dy\right)dr
    =\int_{1-\frac1x}^1\om(r)(1-r)^\beta r^\frac{1}{1-r}\,dr
		\lesssim\frac{\om_x}{x^{\beta}},
    \end{split}
    \end{equation*}
we deduce that these terms do not play any role, and hence all the information is contained in $I_2(x)$. Since $I_2$ is comparable to the left-hand side of (iv) by \eqref{eq:I2-calc1}, it follows that (iii) and (iv) are equivalent. This completes the proof of the lemma.
\end{proof}

With the aid of Lemma~\ref{M-lemma} and the results already obtained in the previous sections we can prove Theorem~\ref{thm:A2-weight-M}.

\medskip

\noindent{\emph{Proof of Theorem~\ref{thm:A2-weight-M}.}} (i) implies (ii) by testing with the monomial $m_n$, and (ii) implies (i) by Parseval's identity. It remains to show that (ii) and (iii) are equivalent. But the proof of Proposition~\ref{proposition-norm} shows that (ii) implies (iii), provided $1<q<\infty$. To deal with the case $q=1$ observe that an integration by parts gives
    \begin{equation*}
    \om_x=\int_0^1\widehat\om(s)s^x\frac{x}{s}\,ds\ge\int_0^{1-\frac1x}s^x\widetilde\om(s)(1-s)\frac{x}{s}\,ds
    \ge\int_0^{1-\frac1x}s^x\widetilde\om(s)\frac{ds}{s}
    \ge\int_0^{1-\frac1x}s^x\widetilde\om(s)\,ds
    \end{equation*}
for each radial weight $\om$, and thus $\om\in\M$ by Lemma~\ref{M-lemma}(iv) with $\beta=0$.

It remains to prove that (iii) implies (ii). To do this, we first observe that Lemma~\ref{M-lemma}(iv) and Lemma~\ref{M-lemma-1}(iii), with $\b=q-1$, yield
	$$
	\int_{1-\frac1x}^1\widehat{\om}_{[q-2]}(t)\,dt\lesssim\frac{\om_x}{x^{q-1}}\lesssim\left(\om_{[q-1]}\right)_x,\quad 1\le x<\infty,
	$$
and hence it suffices to show that
	$$
	I(x)=\int_0^{1-\frac1x}t^{x}\widehat{\om}_{[q-2]}(t)\,dt\lesssim\left(\om_{[q-1]}\right)_x,\quad 1\le x<\infty.
	$$
By Fubini's theorem,
	\begin{equation*}
	\begin{split}
	I(x)
	&=\int_0^{1-\frac1x}\om(t)\left(\int_0^tr^x(1-r)^{q-2}\,dr\right)dt\\
	&\quad+\widehat{\om}\left(1-\frac1x\right)\int_0^{1-\frac1x}r^x(1-r)^{q-2}\,dr
	=I_1(x)+I_2(x),\quad 1\le x<\infty.
	\end{split}
	\end{equation*}
An integration by parts gives
	$$
	J(t,x)=\int_0^{t}r^x(1-r)^{q-2}\,dr
	=\frac{t^{x+1}}{(x+1)}(1-t)^{q-2}+\frac{q-2}{x+1}\int_0^tr^{x+1}(1-r)^{q-3}\,dr,
	$$
and hence, for $1\le q\le2$, we have
	$$
	J(t,x)\le\frac{t^x(1-t)^{q-1}}{x(1-t)}\le t^x(1-t)^{q-1},\quad 0<t\le1-\frac1x.
	$$
In the case $2<q\le3$ we deduce	
	\begin{equation*}
	\begin{split}
	J(t,x)
	&\le t^x(1-t)^{q-1}+\frac1x\int_0^tr^x(1-r)^{q-3}\,dr
	\le t^x(1-t)^{q-1}+\frac{(1-t)^{q-3}}{x}\frac{t^{x+1}}{x+1}\\
	&\le 2t^x(1-t)^{q-1},\quad 0<t\le1-\frac1x.
	\end{split}
	\end{equation*}
In general, if $n\le q\le n+1$ for some $n\in\N$, then after $n-1$ integrations by parts and estimates similar to those performed above, we obtain
	$$
	J(t,x)
	\le nt^x(1-t)^{q-1}
	\le qt^x(1-t)^{q-1},\quad 0<t\le1-\frac1x.
	$$
By implementing this inequality into $I_1$ and $I_2$, we finally obtain
	$$
	I_1(x)\le q\int_0^{1-\frac1x}t^x\om(t)(1-t)^{q-1}\,dt\le q\left(\om_{[q-1]}\right)_x,\quad 1\le x<\infty,
	$$
and
	$$
	I_2(x)\lesssim\widehat{\om}\left(1-\frac1x\right)\frac1{x^{q-1}}\lesssim\frac{\om_x}{x^{q-1}}\lesssim\left(\om_{[q-1]}\right)_x,\quad 1\le x<\infty,
	$$
by Lemma~\ref{M-lemma-1}(iii) with $\b=q-1$. Thus $I(x)=I_1(x)+I_2(x)\lesssim\left(\om_{[q-1]}\right)_x$ for all $1\le x<\infty$, and we are done. \hfill$\Box$

\section{Proof of Theorem~\ref{thm:Hp-Ap-S}}\label{sec:thm4}

The pseudohyperbolic disc centered at $a\in\D$ and of radius $0<r<1$ is the set $\Delta(a,r)=\{z\in\D:|\phi_a(z)|<r\}$, where $\vp_a(z)=(a-z)/(1-\overline{a}z)$. It is well known that $\Delta(a,r)$ is the Euclidean disc $D(A,R)$ centered at $A=\frac{1-r^2}{1-|a|^2r^2}a$ and of radius $R=\frac{1-|a|^2}{1-|a|^2r^2}r$. Therefore, for each fixed $r\in(0,1)$, the Euclidean area of $\Delta(a,r)$ is comparable to $(1-|a|)^2$. Moreover, the Carleson square $S(a)$ induced by $a\in\D\setminus\{0\}$ is the polar rectangle
	$$
	S(a)=\left\{re^{i\theta}\in\D:|a|<r<1,\,|\arg ae^{-i\theta}|<\frac{1-|a|}2\right\}.
	$$

\medskip

\noindent{\emph{Proof of Theorem~~\ref{thm:Hp-Ap-S}.}} (i) If $\widehat{\om}_{[p-2]}\not\in L^1$, then trivially $\|f\|_{H^p_\om}\lesssim\|f'\|_{A^p_{\widehat{\om}_{[p-2]}}}$ for all $f\in\H(\D)$. Assume now $\widehat{\om}_{[p-2]}\in\DD$. The generalization of the Littlewood-Paley formula by Stein \cite{Stein1933}, also known as the Hardy-Spencer-Stein formula, states that
	\begin{equation}\label{eq:hardy-stein-spencer}
  M_p^p(r,f)
	=\frac1{2\pi}\int_{D(0,r)}\Delta|f|^p(u)\log\frac{r}{|u|}\,dA(u)
	+|f(0)|^p,\quad 0<r<1,\quad  f\in\H(\D).
	\end{equation}
This and H\"{o}lder's inequality imply
	\begin{equation}\label{eq:hardystein1}
	\begin{split}
  \frac1{2\pi}\int_{D(0,r)}\Delta|f|^p(u)\log\frac{r}{|u|}\,dA(u)
	&=M_p^p(r,f)-|f(0)|^p
	\le M_2^p(r,f)-|f(0)|^p\\
	&\le\left(\frac2{\pi}\int_{D(0,r)}|f'(u)|^2\log\frac{r}{|u|}\,dA(u)\right)^\frac{p}{2}\\
	&\lesssim\sup_{u\in D(0,r)}\left(\left|f'(u)\right|^{p}\left(1-|u|^{2}\right)^{p}\right)
	\end{split}
	\end{equation}
for any fixed $0<r<1$. Since there exists a $\rho\in(0,1)$ such that $\Delta(u,\r)\subset D(0,\frac{1+r}{2})$ for every $u\in D(0,r)$, the subharmonicity of $|f'|^p$ yields
	\begin{equation}\label{eq:subharmonic}
	\sup_{u\in D(0,r)}\left(\left|f'(u)\right|^{p}\left(1-|u|^{2}\right)^{p}\right)
	\lesssim\int_{D\left(0,\frac{1+r}{2}\right)}\left|f'(w)\right|^{p}\left(1-|w|^{2}\right)^{p-2}\,dA(w).
	\end{equation}
By combining \eqref{eq:hardystein1} and \eqref{eq:subharmonic}, we deduce
	$$
	\int_{D(0,r)}\Delta|f|^p(u)\log\frac{r}{|u|}\,dA(u)
	\lesssim\int_{D\left(0,\frac{1+r}{2}\right)}\left|f'(w)\right|^{p}\left(1-|w|^{2}\right)^{p-2}\,dA(w).
	$$
An application of this to $f\circ\varphi_{\zeta}$ yields
	\begin{equation}\label{eq:changingvariable2}
	\int_{\Delta(\zeta,r)}\Delta|f|^p(z)\log\frac{r}{\left|\varphi_{\zeta}(z)\right|}\,dA(z)
	\lesssim\int_{\Delta\left(\zeta,\frac{1+r}{2}\right)}\left|f'(z)\right|^{p}\left(1-|z|^{2}\right)^{p-2}\,dA(z).
	\end{equation}
For each $\zeta\in\D$, write $\zeta^\star=\min\{|z|:z\in\Delta(\zeta,r)\}$, and observe that $1-|\zeta^\star|\asymp1-|\zeta|$ for all $\zeta\in\D$. Moreover, for each $a\in\D$, close enough to the boundary, there exists $b=b(a,r)\in\D$ such that $\cup_{z\in S(a)}\Delta\left(z,\frac{1+r}{2}\right)\subset S(b)$, $\arg b=\arg a$ and $1-|b|\asymp1-|a|$, where the constants of comparison do not depend on $a$. Hence
	\begin{equation*}
	\begin{split}
	&\int_{S(a)}(1-|z|)^{p-2}\left(\int_{\Delta\left(z,\frac{1+r}{2}\right)}\frac{\widehat{\om}(\z^\star)}{(1-|\z|)^2}dA(\zeta)\right)\,dA(z)\\
	&\quad\le\int_{S(b)}\frac{\widehat{\om}(\z^\star)}{(1-|\z|)^2}\left(\int_{S(a)\cap\Delta\left(\zeta,\frac{1+r}{2}\right)}(1-|z|)^{p-2}\,dA(z)\right)dA(\zeta)\\
	&\quad\lesssim\int_{S(b)}\widehat{\om}(\z^\star)(1-|\z^\star|)^{p-2}\,dA(\zeta)
	\lesssim\widehat{\om}_{[p-2]}(S(b))
	\lesssim\widehat{\om}_{[p-2]}(S(a)),\quad |a|\to1^-,
	\end{split}
	\end{equation*}
by the hypothesis $\widehat{\om}_{[p-2]}\in\DD$. This together with Lemma~\ref{AuxLemmaEmbedding}, Fubini's theorem and \eqref{eq:changingvariable2} yields
	\begin{equation*}
	\begin{split}
	\|f'\|_{A^p_{\widehat{\om}_{[p-2]}}}^p
	&\gtrsim\int_\D|f'(z)|^p(1-|z|)^{p-2}\left(\int_{\Delta\left(z,\frac{1+r}{2}\right)}\frac{\widehat{\om}(\z^\star)}{(1-|\z|)^2}\,dA(\zeta)\right)\,dA(z)\\
	&\gtrsim\int_\D\frac{\widehat{\om}(\z^\star)}{(1-|\z|)^2}\left(\int_{\Delta(\z,r)}\Delta|f|^p(z)\log\frac{r}{|\varphi_\z(z)|}\,dA(z)\right)dA(\zeta)\\
	&\gtrsim\int_\D\Delta|f|^p(z)\frac{\widehat{\om}(z)}{(1-|z|)^2}\left(\int_{\Delta(z,r)}\log\frac{r}{|\varphi_z(\zeta)|}\,dA(\z)\right)dA(z)\\
	&\asymp\int_\D\Delta|f|^p(z)\widehat{\om}(z)\,dA(z)
	=\|f\|_{H^p_\om}^p,\quad f\in\H(\D).
	\end{split}
	\end{equation*}
	
Conversely, assume that $\|f\|_{H^p_\om}\lesssim\|f'\|_{A^p_{\widehat{\om}_{[p-2]}}}$ for all $f\in\H(\D)$, and $\widehat{\om}_{[q-2]}\in L^1$. By testing this inequality with the monomial $m_n$ gives $n^{2-p}\left(\left(\widehat{\om}_{[p-2]}\right)_{[2-p]}\right)_{np-1}\lesssim\left(\widehat{\om}_{[p-2]}\right)_{np-1}$ for all $n\in\N$. Let $x\ge p-1$, and choose $n\in\N$ such that $np-1\le x<(n+1)p-1$. Then
	\begin{equation*}
	\begin{split}
	x^{2-p}\left(\left(\widehat{\om}_{[p-2]}\right)_{[2-p]}\right)_{x}
	&\le((n+1)p-1)^{2-p}\left(\left(\widehat{\om}_{[p-2]}\right)_{[2-p]}\right)_{np-1}\\
	&\le(2p)^{2-p}n^{2-p}\left(\left(\widehat{\om}_{[p-2]}\right)_{[2-p]}\right)_{np-1}
	\lesssim\left(\widehat{\om}_{[p-2]}\right)_{np-1}\\
	&\le\left(\widehat{\om}_{[p-2]}\right)_{x-p}
	\lesssim\left(\widehat{\om}_{[p-2]}\right)_{x},\quad p-1\le x<\infty,
	\end{split}
	\end{equation*}
and thus $x^{2-p}\left(\left(\widehat{\om}_{[p-2]}\right)_{[2-p]}\right)_{x}\lesssim\left(\widehat{\om}_{[p-2]}\right)_{x}$ for all $1\le x<\infty$. This in turn is equivalent to $\widehat{\om}_{[p-2]}\in\DD$ by Lemma~\ref{D-hat-lemma}(iii). Thus (i) is proved.

(ii) Assume first $\om\in\DD$. In this part of the proof we begin with using an idea from the proof of \cite[Theorem~A]{Lu88}. By \eqref{eq:hardy-stein-spencer} we have
	\begin{equation*}
	\begin{split}
	|f'(0)|^p
	&\le(\|f\|_{H^2}^2 - |f(0)|^2)^{\frac{p}{2}}
	\le\|f\|_{H^2}^p-|f(0)|^p
	\le\|f\|_{H^p}^p-|f(0)|^p\\
	&=\frac{1}{2\pi}\int_{\D}\Delta|f|^{p}(u)\log\frac{1}{|u|}~dA(u).
	\end{split}
	\end{equation*}
Since $p>2$ by the hypothesis, $\Delta|f|^p$ is subharmonic in $\D$, and hence its $L^1$-mean is non-decreasing. It follows that the last integral is comparable to the same integral where $-\log|u|$ is replaced by $(1-|u|)$. Thus
	\begin{equation*}
	\begin{split}
	|f'(0)|^p
	\lesssim\int_{\D}\Delta|f|^p(u)(1-|u|)~dA(u)
	\le\int_{\D}\Delta|f|^p(u)\,dA(u),\quad f\in\H(\D).
	\end{split}
	\end{equation*}
An application of this inequality to $f_r$, where $0<r<1$ is fixed, gives
	\begin{equation*}
	\begin{split}
	|f'(0)|^p
	\lesssim\int_{D(0,r)}\Delta|f|^p(u)\,dA(u),\quad f\in\H(\D).
	\end{split}
	\end{equation*}
Apply now this inequality to $f\circ\vp_\zeta$ to deduce
	\begin{equation*}
	\begin{split}
	|f'(\zeta)|^p(1-|\zeta|^2)^p
	&\lesssim\int_{\Delta(\zeta,r)}\Delta|f|^p(z)\,dA(z),\quad f\in\H(\D).
	\end{split}
	\end{equation*}
This together with Fubini's theorem, Lemma~\ref{D-hat-lemma}(ii) and \eqref{Hclassdef} yields
	\begin{equation}\label{dirichlet1}
	\begin{split}
	\|f'\|_{A^p_{\widehat{\om}_{[p-2]}}}^p
	&\lesssim\int_{\D}\left(\int_{\Delta(\z,r)}\Delta|f|^p(z)\,dA(z)\right)\frac{\omg(\z)}{(1-|\z|^2)^2}\,dA(\z)\\
	&=\int_{\D}\Delta|f|^p(z)\left(\int_{\Delta(z,r)}\frac{\omg(\z)}{(1-|\z|)^2}\,dA(\z)\right)\,dA(z)
	\asymp\|f\|_{H^p_\om}^p,\quad f\in\H(\D).
	\end{split}
	\end{equation}

Conversely, assume that $\|f'\|_{A^p_{\widehat{\om}_{[p-2]}}}\lesssim\|f\|_{H^p_\om}$ for all $f\in\H(\D)$. By testing this inequality with the monomial $m_n$ gives $n^{p-2}\left(\widehat{\om}_{[p-2]}\right)_{np-1}\lesssim\widehat{\om}_{np-1}$ for all $n\in\N$. This implies $x^{p-2}\left(\widehat{\om}_{[p-2]}\right)_{x}\lesssim\widehat{\om}_{x}$ for all $1\le x<\infty$, which is in turn equivalent to $\widehat{\om}\in\DD$ by Lemma~\ref{D-hat-lemma}(iii). This yields $\om\in\DD$ by Lemma~\ref{D-hat-lemma}(v). Therefore (ii) is now proved.

To prove the assertions for the class $\SSS$, we appeal to the Growth and Distortion theorems for functions in $\SSS$. Namely, by \cite[Theorem~2.7]{Duren83},
	\begin{equation}\label{Eq:growth-distortion}
	\left|\frac{f'(z)}{f(z)}\right|\le\frac1{|z|}\frac{1+|z|}{1-|z|},\quad z\in\D\setminus\{0\},\quad f\in\SSS,
	\end{equation}
and hence, for $2\le p<\infty$, we have
	\begin{equation*}
	\begin{split}
	M_p^p(r,f')
	&=\frac1{2\pi}\int_0^{2\pi}|f(re^{i\theta})|^{p-2}|f'(re^{i\theta})|^2\left|\frac{f'(re^{i\theta})}{f(re^{i\theta})}\right|^{p-2}\,d\theta
	\lesssim\frac{M_1(r,\Delta|f|^p)}{r^{p-2}(1-r)^{p-2}},\quad 0<r<1.
	\end{split}
	\end{equation*}
The assertion for $2\le p<\infty$ follows from this estimate and standard arguments. In a similar way we obtain
	$$
	M_p^p(r,f')\gtrsim M_1(r,\Delta|f|^p)r^{2-p}(1-r)^{2-p},\quad 0<r<1,
	$$
provided $0<p<2$. This yields the assertion concerning the class $\SSS$. \hfill$\Box$

\section{Proofs of Theorems~\ref{Thm1} and \ref{Thm2}}\label{Sec:proofs of theorems on conf maps}

We begin with finding an upper bound for $J^p_\om$ in terms of the norm in a suitable weighted Dirichlet space.

\begin{proposition}\label{prop:Jp-Ap-om}
Let $\om$ be a radial weight. Then the following statements hold:
\begin{itemize}
\item[\rm(i)] If $0<p\le1$, then $J^p_\om(f)\lesssim\|f'\|_{A^p_{\widehat{\om}_{[p-2]}}}^p$ for all $f\in\SSS$;
\item[\rm(ii)] If $0<p<1$ and $\om\in\DD$, then $J^p_\om(f)\lesssim\|f'\|_{A^p_{\widehat{\om}_{[p-2]}}}^p$ for all $f\in\H(\D)$;
\item[\rm(iii)] $J^1_\om(f)\lesssim J^1_{\widehat{\om}}(f')+|f(0)|\lesssim\|f'\|_{A^1_\om}+|f(0)|$ for all $f\in\H(\D)$;
\item[\rm(iv)] If $1<p<\infty$ and $\om\in\Dd$, then
	\begin{equation}\label{eq:Jp-Ap-f'}
	J^p_\om(f)\lesssim\|f'\|^p_{A^p_{\om_{[p-1]}}}+|f(0)|^p,\quad f\in\H(\D).
	\end{equation}
Moreover, this estimate is in general false for the class $\SSS$ if $\om\in\DD\setminus\DDD$.
\end{itemize}
\end{proposition}

\begin{proof} (i) Observe that
    \begin{equation}\label{19}
    \begin{split}
    |f(re^{i\t})|^p
    =\left|\int_0^{re^{i\t}}f'(\z)\,d\z+f(0)\right|^p
    \le2^{p}\left(\left(\int_0^r|f'(te^{i\t})|\,dt\right)^p+|f(0)|^p\right),\quad f\in\H(\D),
    \end{split}
    \end{equation}
for each $0<p<\infty$, and hence
	\begin{equation}\label{eq:Minf-std-ieq}
	M_\infty^p(r,f)\lesssim\left(\int_0^r M_\infty(t,f')\,dt\right)^p+|f(0)|^p,\quad f\in\H(\D).
	\end{equation}
Denote $r_n=\max\{0,1-2^n(1-r)\}$ for all $n\in\mathbb{Z}$ so that $r_n=0$ for all $n\ge-\log_2(1-r)$ and
    $$
    \cdots<r_2<r_1<r_0=r<r_{-1}=\frac{1+r}{2}<r_{-2}<\cdots<r_{-k}\to1^-,\quad k\to\infty.
    $$
Then, by using the subadditivity of $x^p$ for $0<p<1$, we deduce
    \begin{equation}\label{eq:Minf-subadd1}
    \begin{split}
    \left(\int_0^rM_\infty(t,f')\,dt\right)^p
    &\le\left(\sum_{n\ge0}M_\infty(r_n,f')(r_n-r_{n+1})\right)^p\\
    &\le\sum_{n\ge0}M^p_\infty(r_n,f')2^{np}(1-r)^p\frac{\int_{r_{n}}^{r_{n-1}}(1-s)^{p-1}\,ds}{2^{np}(1-r)^p\frac{1-2^{-p}}{p}}\\
    &\le\frac{p}{1-2^{-p}}\int_0^{\frac{1+r}{2}}M^p_\infty(s,f')(1-s)^{p-1}\,ds,\quad f\in\H(\D).
    \end{split}
    \end{equation}
The estimates \eqref{eq:Minf-std-ieq} and \eqref{eq:Minf-subadd1} together with Fubini's theorem yield
		\begin{equation}\label{ieq:Jpom-calc-cov6}
    \begin{split}
    J_\om^p(f)
    &\lesssim\int_0^1\left(\int_0^{\frac{1+r}{2}}M_\infty^p(t,f')(1-t)^{p-1}\,dt\right)\om(r)\,dr+|f(0)|^p\\
		&=\widehat{\om}(0)\int_0^{\frac12}M_\infty^p(t,f')(1-t)^{p-1}\,dt\\
		&\quad+\int_{\frac12}^1M_\infty^p(t,f')(1-t)^{p-1}\widehat{\om}\left(2t-1\right)\,dt+|f(0)|^p\\
		&\lesssim\int_{\frac12}^1M_\infty^p(t,f')(1-t)^{p-1}\widehat{\om}\left(2t-1\right)\,dt+|f(0)|^p,\quad f\in\H(\D).
		\end{split}
    \end{equation}
It is well known that
	\begin{equation}\label{eq:Minf-S-asymp-phdisks}
	M_\infty(r,f')\asymp M_\infty(\r,f'),\quad 0\le r\le\r\le\frac{1+r}{2}<1,\quad 0\le r<1,\quad f\in\SSS,
	\end{equation}
see \cite[(11) on p.~128]{Pommerenke1} for details. This together with \eqref{ieq:Jpom-calc-cov6} yields
		\begin{equation*}
    \begin{split}
    J_\om^p(f)
    &\lesssim\int_{\frac12}^1M_\infty^p(2t-1,f')(1-t)^{p-1}\widehat{\om}\left(2t-1\right)\,dt\\
		&\asymp\int_{\frac12}^1M_\infty^p(t,f')\widehat{\om}_{[p-1]}(t)\,dt
		\asymp\int_{\frac12}^1M_\infty^p(2t-1,f')\widehat{\om}_{[p-1]}(t)\,dt\\
		&\asymp\int_{0}^1M_\infty^p(t,f')(1-t)^{p-1}\widehat{\om}\left(\frac{1+t}{2}\right)\,dt,\quad f\in\SSS.
		\end{split}
    \end{equation*}
Since Fubini's theorem shows that
	\begin{equation}\label{eq:omhatp-fubinicalc-67}
  \begin{split}
	\int_r^1\left(\int_t^1\widehat{\om}_{[p-2]}(s)\,ds\right)dt
	&=\int_r^1\widehat{\om}_{[p-2]}(s)(s-r)\,ds
	\ge\int_{\frac{1+3r}{4}}^{\frac{1+r}{2}}\widehat{\om}_{[p-2]}(s)(s-r)\,ds\\
	&\ge\widehat{\om}\left(\frac{1+r}{2}\right)\left(1-\frac{1+3r}{4}\right)^{p-2}\left(\frac{1+3r}{4}-r\right)\frac{1-r}{4}\\
	&\asymp\widehat{\om}\left(\frac{1+r}{2}\right)(1-r)^p
	\ge p\int_r^1(1-t)^{p-1}\widehat{\om}\left(\frac{1+t}{2}\right)\,dt,
	\end{split}
  \end{equation}
Lemma~\ref{AuxLemmaEmbedding} with $q=\infty$ yields
	$$
	\int_{0}^1M_\infty^p(t,f')(1-t)^{p-1}\widehat{\om}\left(\frac{1+t}{2}\right)\,dt
	\lesssim\int_{0}^1M_\infty^p(t,f')\left(\int_t^1\widehat{\om}_{[p-2]}(s)\,ds\right)\,dt.
	$$
Hence Fubini's theorem and \eqref{eq:hl} finally give
	\begin{equation}\label{eq:Jp-f'-Ap-closure}
  \begin{split}
	J_\om^p(f)
	&\lesssim\int_{0}^1M_\infty^p(t,f')\left(\int_t^1\widehat{\om}_{[p-2]}(s)\,ds\right)\,dt
	=\int_0^1\widehat{\om}_{[p-2]}(s)\left(\int_0^sM_\infty^p(t,f')\,dt\right)\,ds\\
	&\le\pi\int_0^1M_p^p(s,f')\widehat{\om}_{[p-2]}(s)s\,ds
	=\frac12\|f'\|_{A^p_{\widehat{\om}_{[p-2]}}}^p,\quad f\in\SSS.
	\end{split}
  \end{equation}
Thus the assertion is proved for $0<p<1$. The proof in the case $p=1$ is similar to that above but easier since the argument employed in \eqref{eq:Minf-subadd1} is not needed. Therefore we omit the details of this case.

(ii) In view of \eqref{ieq:Jpom-calc-cov6}, \eqref{eq:Jp-f'-Ap-closure} and Lemma~\ref{AuxLemmaEmbedding} it suffices to show that
	$$
	\int_r^1(1-t)^{p-1}\widehat{\om}\left(2t-1\right)\,dt
	\lesssim\int_r^1\left(\int_t^1\widehat{\om}_{[p-2]}(s)\,ds\right)dt,\quad\frac12\le r<1.
	$$
An appropriate lower estimate for the right hand side is obtained in \eqref{eq:omhatp-fubinicalc-67}. Further, $\om\in\DD$ if and only if $\widehat{\om}_{[p-1]}\in\DD$ by Lemma~\ref{D-hat-lemma}(v). Hence the left hand side satisfies
	\begin{equation*}
  \begin{split}
	\int_r^1(1-t)^{p-1}\widehat{\om}\left(2t-1\right)\,dt	
	&=\frac1{2^p}\int_{2r-1}^1(1-s)^{p-1}\widehat{\om}\left(s\right)\,ds\\
	&\lesssim\int_{\frac{1+r}{2}}^1(1-s)^{p-1}\widehat{\om}\left(s\right)\,ds
	\lesssim\widehat{\om}\left(\frac{1+r}{2}\right)(1-r)^p,
	\end{split}
  \end{equation*}
and the assertion follows.

(iii) The case $p=1$ of \eqref{eq:Minf-std-ieq} and Fubini's theorem yield
	$$
	J^1_\om(f)\lesssim\int_0^1\left(\int_0^rM_\infty(t,f')\,dt\right)\om(r)\,dr+|f(0)|=J_{\widehat{\om}}^1(f')+|f(0)|,
	$$
and thus the first asymptotic inequality is proved. The first estimate above together with \eqref{eq:hl} yields the second inequality.

(iv) Let $x>p-1>0$, where $x=x(\om,p)$ will be fixed later. The estimate \eqref{eq:Minf-std-ieq} and H\"older's inequality yield
	\begin{equation}\label{eq:Jp-Hoelder-asympcalc6x}
	\begin{split}
	J^p_\om(f)
	&\lesssim\int_0^1\left(\int_0^rM_\infty(t,f')(1-t)^\frac{x}{p}(1-t)^{-\frac{x}{p}}\,dt\right)^p\om(r)\,dr+|f(0)|^p\\
	&\le\int_0^1\left(\int_0^rM_\infty^p(t,f')(1-t)^x\,dt\right)\left(\int_0^r\frac{dt}{(1-t)^{\frac{xp'}{p}}}\right)^\frac{p}{p'}\om(r)\,dr+|f(0)|^p\\
	&\asymp\int_0^1\left(\int_0^rM_\infty^p(t,f')(1-t)^x\,dt\right)\frac{\om(r)}{(1-r)^{x+1-p}}\,dr+|f(0)|^p,
	\end{split}
	\end{equation}
where, by Fubini's theorem and \eqref{eq:hl},
	\begin{equation*}
	\begin{split}
	\int_0^rM_\infty^p(t,f')(1-t)^x\,dt
	&=\int_0^rx(1-s)^{x-1}\left(\int_0^sM_\infty^p(t,f')\,dt\right)ds
	+(1-r)^x\int_0^rM_\infty^p(t,f')\,dt\\
	&\lesssim\int_0^rM_p^p(s,f')(1-s)^{x-1}s\,ds
	+M_p^p(r,f')(1-r)^xr.
	\end{split}
	\end{equation*}
It follows that
	\begin{equation*}
	\begin{split}
	J^p_\om(f)
	&\lesssim\int_0^1\left(\int_0^rM_p^p(s,f')(1-s)^{x-1}s\,ds\right)\frac{\om(r)}{(1-r)^{x+1-p}}\,dr
	+\|f'\|^p_{A^p_{\om_{[p-1]}}}+|f(0)|^p\\
	&=\int_0^1M_p^p(s,f')(1-s)^{x-1}\left(\int_s^1\frac{\om(r)}{(1-r)^{x+1-p}}\,dr\right)s\,ds+\|f'\|^p_{A^p_{\om_{[p-1]}}}+|f(0)|^p.
	\end{split}
	\end{equation*}
Since $\om\in\Dd$ by the hypothesis, an integration by parts together with Lemma~\ref{Lemma:weights-in-R} shows that
	\begin{equation}\label{eq:om-1minusr-intest}
	\int_t^1\frac{\om(r)}{(1-r)^{x+1-p}}\,dr\lesssim\widehat{\om}(t)(1-t)^{p-1-x},\quad 0\le t<1,
	\end{equation}
if $x=x(\om,p)>p-1$ is sufficiently small, that is, $x<p-1+\beta$, where $\beta=\beta(\om)>0$ is that of Lemma~\ref{Lemma:weights-in-R}(ii). This  estimate together with Proposition~\ref{proposition-norm}(iv) now yields \eqref{eq:Jp-Ap-f'}.

It remains to show that the estimate \eqref{eq:Jp-Ap-f'} is in general false for the class $\SSS$ if $1<p<\infty$ and $\om\in\DD\setminus\DDD$. To see this, consider the weight $v_\alpha\in\DD\setminus\DDD$ defined by \eqref{V}, where $1<\alpha<\infty$, and the function $f(z)=-\log(1-z)$ which belongs to~$\SSS$. Then
	$$
	J^p_{v_\alpha}(f)\asymp\int_0^1\frac{dr}{(1-r)\left(\log\frac{e}{1-r}\right)^{\alpha-p}}
	$$
and
	$$
	\|f'\|_{A^p_{{v_\alpha}_{[p-1]}}}^p\asymp\int_0^1(1-r)^{1-p+p-1}v_\alpha(r)\,dr
	=\int_0^1\frac{dr}{(1-r)\left(\log\frac{e}{1-r}\right)^{\alpha}}.
	$$
By choosing $\alpha=p+1$ we obtain the contradiction we are after.
\end{proof}

The next step is to bound $\|f\|_{H^p_\om}$ and $\|f\|_{S^p_\om}$ by $J^p_\om(f)$.

\begin{lemma}\label{lemma:H,S-J-univalent}
Let $0<p<\infty$ and $\om$ be a radial weight. Then
	\begin{equation}\label{Eq:H^p}
	\|f\|_{H^p_\om}^p\le2\pi pJ^p_\om(f),\quad f\in\U,
	\end{equation}
and
	\begin{equation}\label{Eq:S^p}
	\|f\|_{S^p_\om}^p\le\pi^\frac{p}{2}J^p_\om(f),\quad f\in\U.
	\end{equation}
None of the corresponding asymptotic inequalities is valid for all $f\in\H(\D)$ unless $\om$ vanishes almost everywhere on $\D$.
\end{lemma}

\begin{proof}
Let $0<p<\infty$ and $\om$ be a radial weight. If $f\in\U$, then
	\begin{equation}\label{lapinftyeq1}
	\begin{split}
	\frac{1}{2\pi p}\int_{D(0,r)}\Delta|f|^{p}(z)\,dA(z)
	&\le\frac{p}{2\pi}\int_{|w|\le M_{\infty}(r,f)}|w|^{p-2}\,dA(w)\\
	&=p\int_{0}^{M_\infty(r,f)}t^{p-1}\,dt
	=M_{\infty}^{p}(r,f).
	\end{split}	
	\end{equation}
The first assertion \eqref{Eq:H^p} follows from this inequality. Moreover, the case $p=2$ in \eqref{lapinftyeq1} yields
	\begin{equation*}
	\begin{split}
	\frac{1}{\pi}\int_{D(0,r)}|f'(z)|^2\,dA(z)\le M_\infty^2(r,f),
	\end{split}
	\end{equation*}
from which \eqref{Eq:S^p} follows.

The monomial $m_n$ satisfies
	$$
	J^p_\om(m_n)=\om_{np},\quad
	\|m_n\|_{H^p_\om}^p=2\pi np\om_{np},\quad\textrm{and}\quad
	\|m_n\|_{S^p_\om}^p=\left(\pi n\right)^{\frac{p}{2}}\om_{np},\quad n\in\N,
	$$
and hence $J^p_\om(m_n)$ does not dominate $\|m_n\|_{H^p_\om}^p$ nor $\|m_n\|_{S^p_\om}^p$ unless $\om$ vanishes almost everywhere on $\D$.
\end{proof}

For $f\in\H(\D)$ with Maclaurin series expansion $f(z)=\sum_{n=0}^\infty\widehat{f}(n)z^n$, write
	$$
	P(r,f)=\sum_{n=1}^\infty|\widehat{f}(n)|r^n,\quad 0\le r<1.
	$$
The next result is a generalization of \cite[Theorem~15]{H-L1926} and \cite[Proposition~2]{M-P1983} to doubling weights.

\begin{lemma}\label{P-lemma}
Let $0<p<\infty$ and $\om\in\DDD$. Then
	\begin{equation}\label{Eq:poiu}
	J^p_\om(f)\lesssim\int_0^1P(r,f)^p\om(r)\,dr+|f(0)|^p
	\lesssim\|f\|_{S^p_\om}^p+|f(0)|^p,\quad f\in\H(\D).
	\end{equation}
Moreover, the estimate $J^p_\om(f)\lesssim\|f\|_{S^p_\om}^p+|f(0)|^p$ is in general false for the class $\SSS$ if $0<p<\infty$ and $\om\in\DD\setminus\DDD$.
\end{lemma}

\begin{proof}
We may assume that $\widehat{\om}(0)=1$: if this is not the case, consider $\om/\widehat{\om}(0)$ instead of $\om$, and cancel the constant $\widehat{\om}(0)$ at the end of the proof. Let $K>1$ and define the sequence $\{r_n\}_{n=0}^\infty=\{r_n(\om,K)\}_{n=0}^\infty$ by the identity $\widehat{\omega}(r_n)=K^{-n}$ for all $n\in\N\cup\{0\}$. Then $r_0=0$ and $r_n\to1^-$, as $n\to\infty$. For $x\in[0,\infty)$, let $E(x)\in\N\cup\{0\}$ such that $E(x)\le x<E(x)+1$, and set $M_n=E\left(\frac{1}{1-r_n}\right)$. By following \cite{P-R2013}, write
	\begin{equation*}
	\begin{split}
	I(0)=I_{\omega,K}(0)=\{k\in\N\cup\{0\}: k<M_1\}
	\end{split}
	\end{equation*}
and
	\begin{equation*}
	\begin{split}
	I(n)=I_{\omega,K}(n)=\{k\in\N:M_n\le k<M_{n+1}\},\quad n\in\N.
	\end{split}
	\end{equation*}
An application of Lemma~\ref{D-hat-lemma}(ii) shows that for $K=C4^\beta$ we have
		$$
		\frac{M_{n+1}}{M_n}\ge r_{n+1}\frac{1-r_n}{1-r_{n+1}}\ge \frac12\left(\frac{K}{C}\right)^\frac1\b=2,\quad r_{n+1}\ge\frac12,
		$$
and similarly, $M_{n+1}/M_n\le2\cdot4^\frac\b\a$ whenever $r_{n}\ge\frac12$. It follows that $M_n\asymp M_{n+1}$ for all $n\in\N\cup\{0\}$.
	
By \cite[Proposition~9]{P-R2015}, see also \cite[Proposition~9]{P-R2013}, for each $0<p<\infty$, $1<K<\infty$ and $\om\in\DD$ with $\widehat{\om}(0)=1$, there exists a constant $C=C(p,\om,K)>0$ such that
    \begin{equation}\label{j13}
    \frac{1}{C}\sum_{n=0}^\infty K^{-n}\left(\sum_{k\in I_{\om,K}(n)}a_k\right)^p
		\le\int_{0}^1g(r)^p\om(r)\,dr
		\le C\sum_{n=0}^\infty K^{-n}\left(\sum_{k\in I_{\om,K}(n)}a_k\right)^p
    \end{equation}
for all $g(r)=\sum_{k=0}^\infty a_k r^k$, where $a_k\ge
0$ for all $k\in\N\cup\{0\}$. The trivial inequality $M_\infty(r,f)\le P(r,f)$ together with \eqref{j13} applied to $P$ gives
	\begin{equation}\label{yte1}
	\begin{split}
	J_{\omega}^p (f)
	&\lesssim\int_0^1P(r,f)^p\omega(r)\,dr+|\widehat{f}(0)|^p
	\asymp \sum_{n=0}^{\infty}K^{-n}\left(\sum_{k\in{I_{\omega,K}(n)}}|\widehat{f}(k)|\right)^p.
	\end{split}
	\end{equation}
The Cauchy-Schwarz inequality and $M_n\asymp M_{n+1}$ now yield
	\begin{equation}\label{yte2}
	\begin{split}
	\sum_{n=0}^{\infty}K^{-n}\left(\sum_{k\in{I_{\omega,K}(n)}}|\widehat{f}(k)|\right)^p
	&\le\sum_{n=0}^{\infty}K^{-n}\left(M_{n+1}-M_n\right)^{\frac{p}{2}}
	\left(\sum_{k\in{I_{\omega,K}(n)}}|\widehat{f}(k)|^2\right)^{\frac{p}{2}}\\
	&\lesssim\sum_{n=0}^{\infty}K^{-n}\left(\sum_{k\in{I_{\omega,K}(n)}}k|\widehat{f}(k)|^2\right)^{\frac{p}{2}}.
	\end{split}
	\end{equation}
On the other hand, Parseval's identity gives
	\begin{equation}\label{parseval1}
	\begin{split}
	\int_{D(0,r)}|f'(z)|^2\,dA(z)
	&=2\pi\sum_{k=1}^\infty k^2|\widehat{f}(k)|^2\int_0^rs^{2k-1}\,ds
	=\pi\sum_{k=1}^\infty k|\widehat{f}(k)|^2r^{2k}.
	\end{split}
	\end{equation}
The proof of \eqref{j13} shows that the first inequality remains valid if we replace $g(r)$ by $g(r^2)$. Therefore \eqref{parseval1} yields
	\begin{equation}\label{yte3}
	\begin{split}
	\|f\|_{S_\omega^p}^p
	&=\int_0^1\left(\pi\sum_{k=1}^\infty k|\widehat{f}(k)|^2r^{2k}\right)^{\frac{p}{2}}\omega(r)\,dr
	\gtrsim\sum_{n=0}^{\infty}K^{-n}\left(\sum_{k\in{I_{\omega,K}(n)},k\in\N}k|\widehat{f}(k)|^2\right)^{\frac{p}{2}}.
	\end{split}
	\end{equation}
The inequality \eqref{Eq:poiu} follows by combining \eqref{yte1}, \eqref{yte2} and \eqref{yte3}.

It remains to show that the estimate $J^p_\om(f)\lesssim\|f\|_{S^p_\om}^p+|f(0)|^p$ is in general false for $0<p<\infty$ and $f\in\SSS$ if $\om\in\DD\setminus\DDD$. To see this, consider the weight $v_\alpha\in\DD\setminus\DDD$ defined by \eqref{V}, where $1<\alpha<\infty$, and the function $f(z)=-\log(1-z)$ which belongs to $\SSS$. Then
	\begin{equation}\label{eq:Jp-valpha-asymp}
	J^p_{v_\alpha}(f)=\int_0^1P(r,f)^p\om(r)\,dr\asymp\int_0^1\frac{dr}{(1-r)\left(\log\frac{e}{1-r}\right)^{\alpha-p}}
	\end{equation}
and
	$$
	\|f\|_{S^p_{v_\alpha}}^p\asymp\int_0^1\frac{dr}{(1-r)\left(\log\frac{e}{1-r}\right)^{\alpha-\frac{p}{2}}}.
	$$
By choosing $\alpha=p+1$ we obtain the desired contradiction.
\end{proof}

The last auxiliary result in this section generalizes \cite[Theorem~1]{HT1978} and \cite[Theorem~2]{M-P1983} to doubling weights. The proof relies heavily on Lemma~\ref{P-lemma}.

\begin{lemma}\label{Lemma:H-versus-S}
Let $\om\in\DDD$. Then the following statements hold:
\begin{itemize}
\item[(i)] If $0<p\le2$, then $\|f\|_{S_{\omega}^p}\lesssim\|f\|_{H_{\omega}^p}$ for all $f\in\H(\D)$;
\item[(ii)] If $2\le p<\infty$, then $\|f\|_{H_{\omega}^p}\lesssim\|f\|_{S_{\omega}^p}$ for all $f\in\H(\D)$.
\end{itemize}
Both estimates are in general false for the class $\SSS$ if $p\ne2$ and $\om\in\DD\setminus\DDD$.
\end{lemma}

\begin{proof}
Let $0<p<2$ and write $I_{p,\om}(f)=\int_0^1P(r,f)^p\om(r)\,dr$. The trivial estimate $|f(z)|\le P(|z|,f)$ and Jensen's inequality yield
	\begin{equation*}
	\begin{split}
	\|f\|_{H_{\omega}^p}^{\frac{p^2}{2}}
	&\ge\left(\int_0^1I_{p,\om}(f)\frac{\int_{D(0,r)}|f'(z)|^2\,dA(z)}{P(r,f)^2}
	\frac{P(r,f)^p\omega(r)\,dr}{I_{p,\om}(f)} \right)^{\frac{p}{2}}\\
	&\ge\int_0^1\left(I_{p,\om}(f)\frac{\int_{D(0,r)}|f'(z)|^2\,dA(z)}{P(r,f)^2}\right)^{\frac{p}{2}}
	\frac{P(r,f)^p \omega(r) dr}{I_{p,\om}(f)}\\
	&=\left(I_{p,\om}(f)\right)^{\frac{p}{2}-1}\|f\|_{S_{\omega}^p}^{p}.
	\end{split}
	\end{equation*}
Since $I_{p,\om}(f)=\int_0^1P(r,f)^p\om(r)\,dr\lesssim\|f\|_{S^p_\om}^p$ by Lemma~\ref{P-lemma}, we deduce (i).

If $2<p<\infty$, then H\"older's inequality yields
	\begin{equation*}
	\begin{split}
	\|f\|_{H^p_\om}^p
	&\lesssim\int_0^1M_\infty^{p-2}(r,f)\left(\int_{D(0,r)}|f'(z)|^2\,dA(z)\right)\om(r)\,dr\\
	&\le\left(\int_0^1M_\infty^p(r,f)\om(r)\,dr\right)^\frac{p-2}{p}
	\left(\int_0^1\left(\int_{D(0,r)}|f'(z)|^2\,dA(z)\right)^\frac{p}{2}\om(r)\,dr\right)^\frac{2}{p}\\
	&=\left(J^p_\om(f)\right)^\frac{p-2}{p}\|f\|_{S^p_\om}^2,\quad f\in\H(\D).
	\end{split}
	\end{equation*}
The assertion (ii) now follows by Lemma~\ref{P-lemma}.

It remains to prove the assertion for $\om\in\DD\setminus\DDD$ and $p\ne2$. For $f(z)=-\log(1-z)$ and $v_\alpha$ defined by \eqref{V} we have
	\begin{equation}\label{eq:Sp-valpha-asymp}
	\|f\|_{S^p_{v_\alpha}}^p\asymp\int_0^1\frac{dr}{(1-r)\left(\log\frac{e}{1-r}\right)^{\alpha-\frac{p}{2}}}
	\end{equation}
and
	$$
	\|f\|_{H^p_{v_\alpha}}^p\asymp\int_0^1\frac{dr}{(1-r)\left(\log\frac{e}{1-r}\right)^{\alpha+1-p}}.
	$$
The last asymptotic is an easy consequence of the fact that, for each $\delta>0$, the function $x\mapsto x^\delta\log\frac{2e^\frac1\delta}{x}$ is increasing on $(0,2)$. Since $\alpha-\frac{p}{2}\le\alpha+1-p$ if and only if $p\le 2$, the assertions follow.
\end{proof}

With these preparations we can prove Theorems~\ref{Thm1} and~\ref{Thm2}. We first prove the latter one.

\medskip

\noindent{\emph{Proof of Theorem~\ref{Thm2}.}} Let $2\le p<\infty$ and $\om\in\DDD$. Then Theorem~\ref{thm-iff}(iii) yields
	$$
	\|f\|_{D^p_{\om_{[p-1]}}}\asymp \|f\|_{D^p_{\widehat{\om}_{[p-2]}}},\quad f\in \H(\D),
	$$
and Theorem~\ref{thm:Hp-Ap-S} gives
	$$
	\|f\|_{D^p_{\widehat{\om}_{[p-2]}}}^p\lesssim\|f\|_{H^p_{\om}}^p+|f(0)|^p,\quad f \in \H(\D),
	$$
provided $\om\in\DD$. Further, since $\om\in\DDD$ by the hypothesis, Lemma~\ref{Lemma:H-versus-S}(ii) implies
	$$
	\|f\|_{H^p_{\om}}\lesssim\|f\|_{S^p_{\om}},\quad f \in \H(\D).
	$$
Lemma~\ref{lemma:H,S-J-univalent} in turn yields
	$$
	\|f\|_{S^p_{\om}}^p\lesssim J^p_{\om}(f), \quad f \in \U,
	$$
and finally, since $\DDD\subset\Dd$, we have
	$$
	J^p_{\om}(f) \lesssim\|f\|_{D^p_{\om_{[p-1]}}}^p,\quad f\in\H(\D),
	$$
by Proposition~\ref{prop:Jp-Ap-om}(iv). By pulling these estimates together and observing that $|f(0)|^p\widehat{\om}(0)\le J^p_\om(f)$ for all $f\in\H(\D)$, we deduce
	\begin{equation}
	\|f\|_{D^p_{\om_{[p-1]}}}^p
	\asymp \|f\|_{D^p_{\widehat{\om}_{[p-2]}}}^p
	\asymp \|f\|_{H^p_{\om}}^p+|f(0)|^p
	\asymp \|f\|_{S^p_{\om}}^p+|f(0)|^p
	\asymp J^p_{\om}(f),\quad f\in\U.
	\end{equation}

It remains to deal with the quantity $I_{p,q,\om}(f)$. It is well known~\cite[Chapter~5]{Duren70} that for each fixed $0<p\le\infty$ we have
	\begin{equation}\label{Eq:maxmod}
	M_p(r,f')\lesssim\frac{M_p\left(\frac{1+r}{2},f\right)}{1-r},\quad 0<r<1,\quad f\in\H(\D).
	\end{equation}
Another result that we will need is \cite[Proposition~8.1]{Pommerenke2} which states that, for each fixed $0<p<\infty$, we have
	\begin{equation}\label{Eq:laplacian}
	\int_0^{2\pi}\Delta|f|^p(re^{i\theta})\,d\theta\lesssim\frac{M^p_{\infty}(r,f)}{1-r},\quad \frac{1}{2}\leq r<1,\quad f\in\SSS.
	\end{equation}
By applying \eqref{Eq:maxmod}, with $p=\infty$, and \eqref{Eq:laplacian}, with $p=2$, we deduce
\begin{equation*}
\begin{split}
I_{p,q,\om}(f)
&\lesssim\int_{\frac{1}{2}}^1M_\infty(r,f')^{(q-2)\frac{p}{q}}\left(M_2(r,f')\right)^{2\frac{p}{q}}(1-r)^{p(1-\frac1q)}\om(r)\,dr\\
&\lesssim\int_\frac12^1\left(\frac{M_\infty\left(\frac{1+r}{2},f\right)}{1-r}\right)^{(q-2)\frac{p}{q}}\left(\frac{M_\infty^2(r,f)}{1-r}\right)^{\frac{p}{q}}(1-r)^{p(1-\frac1q)}\om(r)\,dr\\
&\le\int_\frac12^1M_\infty^p\left(\frac{1+r}{2},f\right)\om(r)\,dr,\quad f\in\SSS.
\end{split}
\end{equation*}
Since \cite[(11) on p.~128]{Pommerenke1} yields
	$$
	M_\infty(r,f)\asymp M_\infty(\r,f),\quad \frac12\le r\le\r\le\frac{1+r}{2}<1,\quad f\in\SSS,
	$$
we deduce $I_{p,q,\om}(f)\lesssim J^p_\om(f)$ for all $f\in\SSS$. It follows that $I_{p,q,\om}(f)+|f(0)|^p\lesssim J^p_\om(f)$ for all $f\in\U$. This part of the proof is valid for any radial weight $\omega$.

For the converse implication, we first observe that \eqref{eq:Jp-Hoelder-asympcalc6x} and \eqref{eq:om-1minusr-intest} imply
	\begin{equation}\label{Eq:Flett}
	J^p_\om(f)\lesssim\int_0^1M_\infty^p(r,f')(1-r)^p\widetilde{\om}(r)\,dr+|f(0)|^p,\quad f\in\H(\D),
	\end{equation}
provided $1<p<\infty$ and $\om\in\Dd$. Moreover, a careful inspection of the proof of \cite[Theorem~5.9]{Duren70} shows that for $0<\alpha<\beta\le\infty$ there exists a constant $C=C(\alpha,\beta)>0$ such that
    \begin{equation}\label{Eq:Duren1}
    M_\beta(r,g)\le
    CM_\alpha\left(\frac{1+r}{2},g\right)(1-r)^{\frac1\beta-\frac1\alpha},\quad 0\le
    r<1,\quad g\in\H(\D).
    \end{equation}
By combining \eqref{Eq:Flett} and \eqref{Eq:Duren1}, with $\beta=\infty$ and $\alpha=q$, we deduce
	$$
	J^p_\om(r)
	\lesssim\int_0^1M_q^p\left(\frac{1+r}{2},f'\right)(1-r)^{p(1-\frac{1}{q})}\widetilde{\om}(r)\,dr+|f(0)|^p,\quad f\in\H(\D).
	$$
Finally, \eqref{eq:Minf-S-asymp-phdisks} and the proof of Proposition~\ref{proposition-norm}(iv) yield $J^p_\om(r)\lesssim I_{p,q,\om}(f)+|f(0)|^p$. This finishes the proof of the theorem.
\hfill$\Box$

\medskip

\noindent{\emph{Proof of Theorem~\ref{Thm1}.}} By the proof of Theorem~\ref{Thm2} it suffices to consider the case $0<p<2$. First observe that Lemma~\ref{lemma:H,S-J-univalent} implies
	$$
	\|f\|_{H^p_{\om}} \lesssim J^p_{\om}(f) ,\quad f  \in \U,
	$$
for each radial weight $\om$. Further, Lemma~\ref{P-lemma} yields
	$$
	J^p_{\om}(f)\lesssim \|f\|_{S^p_{\om}}^p + |f(0)|^p, \quad f \in \H(\D),
	$$
provided $0<p<\infty$ and $\om\in\DDD$. Finally, by Lemma~\ref{Lemma:H-versus-S}(i) we have
	$$
	\|f\|_{S^p_{\om}}
	\lesssim \|f\|_{H^p_{\om}}\quad f \in \H(\D),
	$$
whenever $\om\in\DDD$ and $0<p<2$. By combining these estimates we deduce the assertion of the theorem in the case $0<p<2$.
\hfill$\Box$

\section{Proof of Theorem~\ref{thm:coefficients}}\label{last section}

We will need one more auxiliary result concerning the class $\M$. It is the next lemma which is an unpublished result by J. A. Pel\'aez and the second author.

\begin{lemma}\label{M-epsilon}
Let $\om\in\M$. Then there exists $\beta=\beta(\om)>0$ such that $\om_{[-\beta]}\in\M$.
\end{lemma}

\begin{proof}
Choose $\beta=\beta(\om)<\log_K C$, where $C=C(\om)>1$ and $K=K(\om)\in\N$ are such that
	$$
  \om_x\ge C \om_{Kx},\quad 1\le x<\infty.
  $$
Let $\{a_n(\beta)\}_{n=0}^\infty$ be positive numbers such that
	$$
	(1-s)^{-\beta}=\sum_{n=0}^\infty a_n(\beta)s^n,\quad 0\le s<1.
	$$
Then
	\begin{equation*}
	\begin{split}
	(\om_{[-\beta]})_x
	=\sum_{n=0}^\infty a_n(\beta)\om_{n+x}
	&\ge C\sum_{n=0}^\infty a_n(\beta)\om_{K(n+x)}\\
	&=C\int_0^1s^{Kx}\sum_{n=0}^\infty a_n(\beta)s^{Kn}\om(s)\,ds\\
	&=C\int_0^1s^{Kx}(1-s^K)^{-\beta}\om(s)\,ds
	\ge\frac{C}{K^\beta}(\om_{[-\beta]})_{Kx},\quad x\ge 1,	
	\end{split}
	\end{equation*}
because $(1-s^K)=(1-s)(1+s+\cdots+s^{K-1})\le K(1-s)$. By choosing $\b>0$ sufficiently small, we obtain $\om_{[-\beta]}\in\M$ by the definition.
\end{proof}

With Lemma~\ref{M-epsilon} in hand we can compare $\sum_{k=1}^\infty|\widehat{f}(k)|^pk^{p-1}\om_{2k}$ and $\|f\|_{S_{\omega}^p}^p$ when $\om\in\M$.

\begin{lemma}\label{Lemma:HL-versus-S}
Let $\om\in\M$. Then the following statements hold:
\begin{itemize}
\item[(i)] If $0<p\le2$, then $\sum_{k=1}^\infty|\widehat{f}(k)|^pk^{p-1}\om_{2k}\lesssim\|f\|_{S_{\omega}^p}^p$ for all $f\in\H(\D)$;
\item[(ii)] If $2\le p<\infty$, then $\|f\|_{S_{\omega}^p}^p\lesssim\sum_{k=1}^\infty|\widehat{f}(k)|^pk^{p-1}\om_{2k}$ for all $f\in\H(\D)$.
\end{itemize}
Both estimates are in general false for the class $\SSS$ if $p\ne2$ and $\om\in\DD\setminus\DDD$.
\end{lemma}

\begin{proof}
Let first $0<p<2$. Then Parseval's identity, H\"older's inequality and Lemma~\ref{M-lemma-1}(iii), with $\beta=1-\frac{p}{2}$, yield
	\begin{equation*}
	\begin{split}
	\|f\|_{S^p_\om}^p
	&=\int_0^1\left(\pi\sum_{k=1}^\infty k|\widehat{f}(k)|^2r^{2k}\right)^{\frac{p}{2}}\omega(r)\,dr
	\ge\pi^\frac{p}{2}\int_0^1\frac{\sum_{k=1}^\infty|\widehat{f}(k)|^pk^{\frac{p}{2}}r^{2k}}{\left(\sum_{k=1}^\infty r^{2k}\right)^\frac{2-p}{2}}\omega(r)\,dr\\
	&\ge\pi^\frac{p}{2}\sum_{k=1}^\infty|\widehat{f}(k)|^pk^{\frac{p}{2}}
	\left(\om_{[1-\frac{p}{2}]}\right)_{2k}
	\gtrsim\sum_{k=1}^\infty|\widehat{f}(k)|^pk^{p-1}\om_{2k},
	\end{split}
	\end{equation*}
and thus (i) is valid.

Let now $2<p<\infty$. By Lemma~\ref{M-epsilon} we may choose $x=x(p,\omega)<\frac2{p'}$ sufficiently large such that $\om_{[xp/2+1-p]}\in\M$. Then Parseval's identity and H\"older's inequality together with standard estimates yield
	\begin{equation*}
	\begin{split}
	\|f\|_{S^p_\om}^p
	&=\pi^\frac{p}{2}\int_0^1\left(\sum_{k=1}^\infty |\widehat{f}(k)|^2k^xr^{\frac{4}{p}k}k^{1-x}r^{\frac{2(p-2)}{p}k}\right)^{\frac{p}{2}}\omega(r)\,dr\\
	&\le\pi^\frac{p}{2}\int_0^1\left(\sum_{k=1}^\infty|\widehat{f}(k)|^pk^{\frac{xp}{2}}r^{2k}\right)
	\left(\sum_{k=1}^\infty k^{(1-x)\frac{p}{p-2}}r^{2k}\right)^{\frac{p-2}{2}}\omega(r)\,dr\\
	&\asymp\sum_{k=1}^\infty|\widehat{f}(k)|^pk^{\frac{xp}{2}}\left(\om_{[xp/2+1-p]}\right)_{2k},
	\end{split}
	\end{equation*}
where
	$$
	\left(\om_{[xp/2+1-p]}\right)_{2k}\lesssim k^{p-1-\frac{xp}{2}}\om_{2k},\quad k\in\N,
	$$
by Lemma~\ref{M-lemma-1}(iii) with $\beta=-\left(\frac{xp}{2}+1-p\right)$. Therefore (ii) is proved.
	
It remains to prove the assertion for $\om\in\DD\setminus\DDD$ and $p\ne2$. For $f(z)=-\log(1-z)$ and $v_\alpha$ defined by \eqref{V} we have \eqref{eq:Sp-valpha-asymp}
and
	\begin{equation}\label{eq:Sp-coeffsum-asymp}
	\sum_{k=1}^\infty|\widehat{f}(k)|^pk^{p-1}(v_\alpha)_{2k}
	\asymp\sum_{k=1}^\infty|\widehat{f}(k)|^pk^{p-1}\widehat{v_\alpha}\left(1-\frac1k\right)
	\asymp\sum_{k=1}^\infty\frac{1}{k\left(\log(k+1)\right)^{\alpha-1}}.
	\end{equation}
The assertion follows by choosing $\alpha$ between 2 and $1+\frac{p}{2}$.
\end{proof}

The statement in the following lemma is trivially true for all radial weights if $p=1$.

\begin{lemma}\label{Thm:coefficients}
Let $1<p<\infty$ and $\om\in\M$. Then
	$$
	J^p_\om(f)\lesssim\sum_{k=0}^\infty|\widehat{f}(k)|^p(k+1)^{p-1}\om_k,\quad f\in\H(\D).
	$$
This estimate is in general false for the class $\SSS$ if $\om\in\DD\setminus\DDD$.
\end{lemma}

\begin{proof}
Let $1<p<\infty$. By Lemma~\ref{M-epsilon} we may choose $x=x(\om,p)>\frac{1-p}{p}$ sufficiently small such that $\om_{[1-p-xp]}\in\M$. Then H\"older's inequality and Lemma~\ref{M-lemma-1}(iii) with $\beta=xp+p-1$ yield
	\begin{equation*}
	\begin{split}
	J^p_{\om}(f)
	&\le\int_0^1\sum_{n=0}^\infty\frac{|\widehat{f}(k)|^pr^k}{(k+1)^{xp}}
	\left(\sum_{k=0}^\infty(k+1)^{xp'}r^k\right)^{p-1}\om(r)\,dr\\
	&\asymp\sum_{k=1}^\infty\frac{|\widehat{f}(k)|^p}{(k+1)^{xp}}\left(\om_{[1-p-xp]}\right)_k
	\lesssim\sum_{k=1}^\infty|\widehat{f}(k)|^p(k+1)^{p-1}\om_k,
	\end{split}
	\end{equation*}
and thus the assertion is proved.

It remains to prove the assertion for $\om\in\DD\setminus\DDD$. For $f(z)=-\log(1-z)$ and $v_\alpha$ defined by \eqref{V} we have \eqref{eq:Jp-valpha-asymp}
and \eqref{eq:Sp-coeffsum-asymp}. The assertion follows by applying Lemma~\ref{D-hat-lemma}(iv) and choosing $\alpha=p+1$.
\end{proof}

We are now in position to prove the last main result of the paper.

\medskip

\noindent{\emph{Proof of Theorem~\ref{thm:coefficients}.}}
By Lemmas~\ref{Thm:coefficients}, \ref{D-hat-lemma}(iv) and~\ref{Lemma:HL-versus-S} we have
	\begin{equation*}
	\begin{split}
	J^p_\om(f)
	\lesssim\sum_{k=0}^\infty|\widehat{f}(k)|^p(k+1)^{p-1}\om_k
	\lesssim\sum_{k=0}^\infty|\widehat{f}(k)|^p(k+1)^{p-1}\om_{2k}
	\lesssim\|f\|_{S^p_\om}^p,\quad f\in\H(\D),
	\end{split}
	\end{equation*}
and the first assertion follows by Theorem~\ref{Thm1}.

To obtain the assertion on the close-to-convex functions, we follow the argument in the proof of \cite[Theorem~5]{HT1978}. By \cite[p.~164]{CP1966} we have
	$$
	|a_k|kr^k\le2M_\infty(r,f)+k^{-\frac12}\left(\frac{1+r}{1-r}\right)^\frac12M_\infty(r,f),\quad 0\le r<1,\quad k\in\N,
	$$
and hence
	\begin{equation}\label{eq:ctcf-ak-Minf-estim}
	|a_k|kr^k\le4M_\infty(r,f),\quad 0\le r\le1-\frac1k,\quad k\in\N.
	\end{equation}
Lemmas~\ref{D-hat-lemma}(ii) and~\ref{Lemma:weights-in-R}(ii) imply $\widehat{\widetilde{\om}}\asymp\widehat{\om}$, and hence Lemma~\ref{AuxLemmaEmbedding} yields $J^p_\om(f)\asymp J^p_{\widetilde{\om}}(f)$ for all $f\in\H(\D)$. This together with \eqref{eq:ctcf-ak-Minf-estim} and Lemma~\ref{Thm:coefficients} gives
	\begin{equation*}
	\begin{split}
	J^p_\om(f)
	&\asymp J^p_{\widetilde\om}(f)
	=\sum_{k=1}^\infty\int_{1-\frac1k}^{1-\frac1{k+1}}M_\infty^p(r,f)\widetilde{\om}(r)\,dr\\
	&\ge2^{-2p-1}\sum_{k=1}^\infty|a_k|^pk^{p-1}\widehat{\om}\left(1-\frac{1}{k+1}\right)\left(1-\frac{1}{k+1}\right)^{kp}\\
	&\asymp\sum_{k=1}^\infty|a_k|^pk^{p-1}\om_k\gtrsim J^p_\om(f),
	\end{split}
	\end{equation*}
provided $1\le p<\infty$ and $\om\in\DDD$. This completes the proof of the theorem.
\hfill$\Box$

\medskip

We complete the section and the paper by the following result which compares $\sum_{k=1}^\infty|\widehat{f}(k)|^pk^{p-1}\om_{k}$ and $\|f'\|^p_{A^p_{\widehat{\om}_{[p-2]}}}$ under natural hypothesis on $\om$ in our setting.

\begin{proposition}\label{lemma:HL-(p-2)}
Let $0<p<\infty$ and let $\om$ be a radial weight such that $\widehat{\omega}_{[p-2]}\in\DD$. Then the following statements hold:
\begin{itemize}
\item[\rm(i)] If $0<p\le2$, then $\sum_{k=1}^\infty|\widehat{f}(k)|^pk^{p-1}\om_{k}\lesssim\|f'\|^p_{A^p_{\widehat{\om}_{[p-2]}}}$ for all $f\in\H(\D)$;
\item[\rm(ii)] If $2\le p<\infty$, then $\|f'\|^p_{A^p_{\widehat{\om}_{[p-2]}}}\lesssim\sum_{k=1}^\infty|\widehat{f}(k)|^pk^{p-1}\om_{k}$ for all $f\in\H(\D)$.
\end{itemize}
\end{proposition}

\begin{proof}
Let first $0<p\le2$. Hardy-Littlewood inequality~\cite[Theorem~6.2]{Duren70} yields
	\begin{equation*}
	\begin{split}
	\|f'\|^p_{A^p_{\widehat{\om}_{[p-2]}}}
	&\gtrsim\sum_{k=1}^\infty |\widehat{f}(k)|^pk^{2p-2}\left(\widehat{\om}_{[p-2]}\right)_{p(k-1)+1}.
	\end{split}
	\end{equation*}
Since $\widehat{\omega}_{[p-2]}\in\DD$ by the hypothesis, Lemma~\ref{D-hat-lemma}(iv)(iii) and Fubini's theorem imply
	$$
	\left(\widehat{\om}_{[p-2]}\right)_{p(k-1)+1}
	\gtrsim\left(\widehat{\om}_{[p-2]}\right)_{k}
	\gtrsim k^{2-p}\widehat{\om}_k
	\asymp k^{1-p}\om_k,\quad k\in\N,
	$$
and the assertion in (i) follows.

Let $2\le p<\infty$ and $\widehat{\omega}_{[p-2]}\in\DD$. Then Lemma~\ref{D-hat-lemma}(v) yields $\om\in\DD$, and hence $\widehat{\om}\in\DD$ by an integration by parts. Hardy-Littlewood inequality~\cite[Theorem~6.3]{Duren70} and Lemma~\ref{D-hat-lemma}(iii), applied to $\widehat{\om}\in\DD$, yield
	\begin{equation*}
	\begin{split}
	\|f'\|^p_{A^p_{\widehat{\om}_{[p-2]}}}
	&\lesssim\sum_{k=1}^\infty |\widehat{f}(k)|^pk^{2p-2}\left(\widehat{\om}_{[p-2]}\right)_{p(k-1)+1}
	\le\sum_{k=1}^\infty |\widehat{f}(k)|^pk^{2p-2}\left(\widehat{\om}_{[p-2]}\right)_{k}\\
	&\lesssim\sum_{k=1}^\infty |\widehat{f}(k)|^pk^{p}\widehat{\om}_{k}
	\asymp\sum_{k=1}^\infty |\widehat{f}(k)|^pk^{p-1}\om_{k},
	\end{split}
	\end{equation*}
and thus the proposition is proved.
\end{proof}

\section{Declarations}

\noindent Ethical Approval: Not applicable.

\noindent Funding: The second author was supported in part by Academy of Finland 356029. 

\noindent Availability of data and materials: Not applicable.

\end{document}